\documentclass[final]{siamltex}

\usepackage{amsmath}
\usepackage{amssymb}
\usepackage{graphicx}
\usepackage{epstopdf}
\usepackage{algorithm}
\usepackage{algorithmic,color}
\usepackage{multirow}
\usepackage{cleveref}
\usepackage{blkarray}% http://ctan.org/pkg/blkarray
\makeatletter
\newcommand*{\rom}[1]{\expandafter\@slowromancap\romannumeral #1@}
\makeatother

\newcommand{\argmin}{{\rm argmin}}
\newcommand{\vect}[1]{\bold{#1}}

\newcommand*\samethanks[1][\value{footnote}]{\footnotemark[#1]}

\newtheorem{Lemma}[theorem]{Lemma}

\title{A two-stage classification method  for high-dimensional data  and point clouds }

\author{Xiaohao Cai\thanks{Mullard Space Science Laboratory (MSSL),  University College London, Surrey RH5 6NT, UK.
  Email:  x.cai@ucl.ac.uk. }
  \and
  Raymond Chan\thanks{Department of Mathematics, City University of Hong Kong,
  Kowloon Tong, Hong Kong. Email:  rchan.sci@cityu.edu.hk.}
  \and
  Xiaoyu Xie\thanks{Department of Mathematics, The Chinese University of Hong Kong,
  Shatin, Hong Kong. Emails: xyxie@math.cuhk.edu.hk and zeng@math.cuhk.edu.hk.}
  \and
  Tieyong Zeng\samethanks[3]
  }

\begin{document}

\maketitle

\begin{abstract}
High-dimensional data classification is a fundamental task in machine learning and imaging science. In this paper, we propose a two-stage multiphase semi-supervised classification method for classifying high-dimensional data and unstructured point clouds. To begin with, a fuzzy classification method such as the standard support vector machine is used to generate a warm initialization.
We then apply a two-stage approach named SaT (smoothing and thresholding) to improve the classification. In the first stage, an unconstraint convex variational model is implemented to purify and smooth the initialization, followed by the second stage which is to project the smoothed partition obtained at stage one to a binary partition. These two stages can be repeated, with the latest result as a new initialization, to keep improving the classification quality. We show that the convex model of the smoothing stage has a unique solution and can be solved by a specifically designed primal-dual algorithm whose convergence is guaranteed.
We test our method and compare it with the state-of-the-art methods on several benchmark data sets. The experimental results demonstrate clearly that our method is superior in both the classification accuracy and computation speed for high-dimensional data and point clouds.
\end{abstract}

\begin{keywords}
Semi-supervised clustering, point cloud classification, variational methods, Graph Laplacian,  SaT (smoothing and thresholding).
\end{keywords}

%\begin{AMS}52A41, 65D15, 68W40, 90C25, 90C90, 68U10\end{AMS}

%-------------------------------------------------------------------
\section{Introduction}\label{sec:introduction}
%-------------------------------------------------------------------
Data sets classification is a fundamental task in remote sensing, machine learning, computer vision, and imaging science
\cite{BM17,YT18,MBYL17,SM00,LCSCL17}. The task, simply speaking, is to group the given data into different classes such that, on one hand,
data points within the same class shares similar characteristics (e.g. distance, edges, intensities, colors, and textures);
on the other hand, pairs of different classes are as dissimilar as possible with respect to certain features. In this paper,  we focus on the task of multi-class semi-supervised classification. The total number of classes $K$ of the
given data sets is assumed to be known, and a few samples, namely the training points, in each class have been labeled. The goal is therefore to infer the labels of the remaining data points using the knowledge of the labeled ones.

For data classification, previous methods are generally based on graphical models, see e.g. \cite{BM17,OWO14,YT18}, and references therein.
In a weighted graph, the data points are vertices and the edge weights signify the affinity or similarity between pairs of data points, where the larger the edge
weight is, the closer or more similar the two vertices are. The basic assumption
for data classification is that vertices in the graph that are connected by edges with large weight should belong to the same class.
Since a fully connected graph is dense and has the size as large as the square of the number of vertices, it is computationally expensive to work on it directly.
In order to circumvent this, some cleverly designed approximations have been developed. For example, in \cite{GMBFP14,MKB13},  spectral approaches are proposed to
efficiently calculate the eigendecomposition of a dense graph Laplacian. In \cite{ELB08,MBBT15}, the nearest neighbor strategy was adopted to build up a sparse graph
where most of its entries are zero, and therefore it is computationally efficient.

In the literature, various studies for semi-supervised classification have been performed by computing a local minimizer of some non-convex energy functional or minimizing a relevant convex relaxation.
To name just a few, we have the diffuse interface approaches using phase field representation based on partial differential equation techniques \cite{BF12,LET12},
the MBO scheme established for solving the diffusion equation \cite{MBO92,MKB13,GMBFP14}, and the variational methods based on graph cut \cite{BM17,YT18}.
In particular, in \cite{BM17}, the convex relaxation models and special constraints on class sizes were investigated. In \cite{YT18}, some novelty region-force terms were introduced
in the variational models to enforce the affinity between vertices and training samples. To the best of our knowledge, all these proposed variational models
have the so-called {\it no vacuum and overlap constraint} on the labeling functions, {which gives rise to non-convex model with NP-hard issues}. By allowing labeling functions to take values in the unit simplex, the original NP-hard combinatorial problem is rephrased into a continuous setting,
see e.g. \cite{BM17,YT18,MBYL17,SM00,YS03,HS11,BTCS13,CCZ13} for various continuous relaxation techniques
(e.g. the ones based on solving the eigenvalue problem, convex approximation, or non-linear optimization) and references therein.

Image segmentation can also be viewed as a special case of the data classification problem \cite{SM00,BF06},
since the pixels in an image can be treated as individual points. Various studies and many algorithms have been considered for
image segmentation. In particular,  variational methods are among the most successful image segmentation techniques,  see e.g. \cite{MS89,CRD07,C15,CCNZ15,DCS10,BCCJKM11,BEVTO07}.
The Mumford-Shah model \cite{MS89}, one of the most important variational segmentation models, was proposed to find piecewise smooth representations of different
segments. It is, however, difficult to solve since the model is non-convex and non-smooth.
Then substantially rich follow-up works were conducted, and many of them considered compromise techniques such as: (i)
simplifying the original complex model, e.g. finding piecewise constant solutions instead of piecewise smooth solutions \cite{CV01-2,LNZS10,VC02}); (ii)
performing convex approximations, e.g. using convex regularization terms like total variation \cite{ROF92,CNCP10}; or (iii) using the smoothing and thresholding (SAT) segmentation methodology \cite{CCZ13,CCSSZ18,CCNZ15,CS13,CYZ13};
for more details refer to e.g. \cite{CEN06,HHMSS12,BCB12,LS12,CEN06,PCCB09,PCBC09a,YBTB10,BEF84} and references therein.
Moreover, various applications were put forward for instance in optical flow \cite{CFNSS14}, tomographic imaging \cite{BCPSS17},
and medical imaging \cite{ZMSM08,CCMS12,CCMS13,BSRT17,SHD15,CSL16}.
%Meanwhile, efficient algorithms which can be used to solve these minimization models were proposed, e.g. the primal-dual algorithm \cite{CP11},
%alternating direction method with multipliers (ADMM) \cite{BPCPE10} and split-Bregman algorithm \cite{GO09}.

In this paper, we propose a multi-class semi-supervised data classification method based on the SaT segmentation methodology
\cite{CCZ13,CCSSZ18,CCNZ15,CS13,CYZ13}. It has been shown to be very promising in terms of segmentation quality and computation speed for images corrupted by many different types of blurs and noises.
Briefly speaking, the SaT methodology includes two main steps: the first step is to obtain a smooth approximation of the given image through minimizing some convex models; and the second step is to get the segmentation results by thresholding the smooth approximation, e.g.
using thresholds determined by the K-means algorithm \cite{KMNPSW02}. Since the models used are convex, the non-convex and NP-hard issues in many existing variational segmentation methods
(e.g. the Mumford-Shah model and piecewise constant Mumford-Shah model mentioned above) were naturally avoided.

Our proposed data classification method mainly contains two stages
with a warm initialization. The warm initialization is a fuzzy classification result which can be generated by any standard classification methods such as the support vector machine (SVM) \cite{CV95}; or by labeling the given data randomly if no proper method is available for the given data (e.g. the data set is too large). Its accuracy is not critical since our proposed method will improve the accuracy significantly from this starting point.

With the warm initialization, the first stage of our method is to find a set of smooth labeling functions, where each gives the probability of every point being in a particular class. They are obtained by minimizing a properly-chosen convex objective functional. In detail, the convex objective functional contains $K$ independent convex sub-minimization
problems, where each corresponds to one labeling function, with no constraints between these $K$ labeling functions.
For each sub-minimization problem, the model is formed by three terms: (i) the data fidelity term restricting the distance between
the smooth labeling function and the initialization; (ii) the graph Laplacian ($\ell_2$-norm) term, and (iii) the total variation ($\ell_1$-norm) built on the graph of the given data.
The graph Laplacian and the total variation terms regularize the labeling functions to be smooth but at the same time close to a representation on the unit simplex.

After obtaining the set of labeling functions, the second stage of our method is just to project the fuzzy classification results obtained at stage one onto the unit simplex to obtain a binary classification result. This step can be done straightforwardly. To improve the classification accuracy, these two stages can be repeated iteratively, where at each iteration the result at the previous iteration is used as a new initialization.

The main advantage of our proposed method is twofold. First, it performs outstandingly in computation speed, since the proposed model is convex and the $K$ sub-minimization
problems are independent with each other (with no constraint on the $K$ labeling functions). The parallelism strategy can be applied straightforwardly to improve computation performance further. On the contrary, the standard start-of-the-art variational data classification methods e.g. \cite{YT18,GMBFP14,LET12} have the constraint on
unit simplex in their minimization models, so that the non-convex or NP-hard issues can affect seriously the efficiency of these methods, even though some convex relaxations may be applied.  Secondly, our method is generally superior in classification accuracy, due to its flexibility of merging the warm initialization and the two-stage iterations which are tractable and manage to improve the accuracy gradually. Note again that we are solving a convex model in the first stage of each iteration, which guarantees a unique global minimizer.
In contrast, there is however no guarantee that the results obtained by the standard start-of-the-art variational data classification methods e.g. \cite{YT18,GMBFP14,LET12} are global
minimizers.
The effectiveness of iterations in our proposed method will be shown in the experiments. For most cases, the clustering accuracy would be increased by a significant margin compared to the first initialization and generally outperforms the state-of-the-art variational classification methods.

The paper is organized as follows. In Section \ref{sec:notation}, we give the basic notation used throughout the paper.
In Section \ref{sec:method}, we present our method for data sets classification. In Section \ref{sec:alg}, we present the algorithm for solving the proposed model and its convergence proof. In Section \ref{sec:numerics},
we test our method on benchmark data sets and compare it with the start-of-the-art methods.
Conclusions are drawn in Section \ref{sec:conclusions}.

%-------------------------------------------------------------------
\section{Basic notation}\label{sec:notation}
%-------------------------------------------------------------------
Let $G = (V, E, w)$ be a weighted undirected graph representing a given point cloud, where $V$ is the vertex set (in which each vertex represents a point) containing $N$ vertices, $E$ is the edge set consisting of pairs of vertices, and $w : E \rightarrow \mathbb{R}_{+}$ is the weight function defined on the edges in $E$. The weights $w(\vect x, \vect y)$ on the edges $(\vect x,\vect y) \in E$ measure the similarity between the two vertices $\vect x$ and $\vect y$; the larger the weight is, the more similar (e.g. closer in distance) the pair of the vertices is.

There are many different ways to define the weight function. Let $d(\cdot, \cdot)$ be a distance metric. Several particularly popular definitions of weight functions are as follows:
(i) radial basis function
\begin{equation}\label{eqn:radial_basis}
w(\vect x,\vect y) := \exp(-d(\vect x,\vect y)^2/(2\xi)), \quad \forall (\vect x, \vect y) \in E,
\end{equation}
for a prefixed constant $\xi >0$;
(ii) Zelnic-Manor and Perona weight function
\begin{equation}\label{eqn:z-m-p_weight}
w(\vect x,\vect y) := \exp(-d(\vect x,\vect y)^2/({\rm var}(\vect x) {\rm var}(\vect y))), \quad \forall (\vect x, \vect y) \in E,
\end{equation}
where ${\rm var}(\cdot)$ denotes the local variance;
and (iii) the cosine similarity
\begin{equation}\label{eqn:cos_weight}
w(\vect x,\vect y) := \dfrac{\langle \vect x,\vect y\rangle}{\sqrt{\langle \vect x,\vect x\rangle\langle \vect y,\vect y\rangle}}, \quad \forall (\vect x, \vect y) \in E,
\end{equation}
where $\langle\cdot,\cdot\rangle$ is the inner product.

Let $W = (w(\vect x,\vect y))_{(\vect x, \vect y)\in E} \in \mathbb{R}^{N\times N}$, the so-called affinity matrix, which is usually assumed to be a symmetric matrix with non-negative entries. Let $D = (h(\vect x,\vect y))_{(\vect x, \vect y)\in E}$ $ \in \mathbb{R}^{N\times N}$ be the diagonal matrix, where its diagonal entries are equal to the sum of the entries on the same row in $W$, i.e.
\begin{equation}
h(\vect x, \vect y) :=
\begin{cases}
\sum_{z\in V} w(\vect x, \vect z), & \vect x = \vect y, \\
0, & {\rm otherwise}.
\end{cases}
\end{equation}
Let ${\vect u} = (u(\vect x))_{\vect x\in V}^\top \in \mathbb{R}^N$, an $N$-length column vector.
Define the graph Laplacian as $L = D - W$, and the gradient operator $\nabla$ on $u(\vect x), \forall \vect x \in V$, as
\begin{equation} \label{eqn:nabla}
\nabla u(\vect x) := \left(w(\vect x,\vect y)(u(\vect x)-u(\vect y))\right)_{(\vect x, \vect y)\in E}.
\end{equation}
Then define the $\ell_1$-norm of an $N$-length vector as
\begin{equation} \label{eqn:nabla-l1}
\|\nabla {\vect u}\|_1 := \sum_{\vect x\in V} |\nabla u(\vect x)| =  \sum_{(\vect x, \vect y)\in E}|w(\vect x,\vect y)(u(\vect x)-u(\vect y))|,
\end{equation}
and the $\ell_2$-norm (also known as Dirichlet energy)
\begin{equation} \label{eqn:nabla-l2}
\|\nabla {\vect u}\|^2 := \dfrac{1}{2}\vect{u}^\top L\vect{u}=\dfrac{1}{2} \sum_{(\vect x, \vect y)\in E}w(\vect x,\vect y)(u(\vect x)-u(\vect y))^2.
\end{equation}
Note, however, that working with
the fully connected graph $E$---like the setting in \eqref{eqn:nabla}, \eqref{eqn:nabla-l1} and \eqref{eqn:nabla-l2}---can be highly computational demanding.

In order to reduce the computational burden, one often only considers the set of edges with large weights. In this paper,
the $k$-nearest-neighbor ($k$-NN) of a point $\vect x$, ${\cal N}(\vect x)$, is used to replace the whole edge set starting from the point $\vect x$ in $E$. Besides the computational saving, one additional benefit of using
$k$-NN graph is its capability to capture local property of points lying close to a manifold.
With the $k$-NN graph, then the definitions in \eqref{eqn:nabla}, \eqref{eqn:nabla-l1} and \eqref{eqn:nabla-l2} become
\begin{gather}
\nabla u(\vect x)  = \left(w(\vect x,\vect y)(u(\vect x)-u(\vect y))\right)_{\vect y\in {\cal N}(\vect x)},  \label{eqn:nabla-n} \\
\|\nabla {\vect u}\|_1  := \sum_{\vect x\in V} |\nabla u(\vect x)| = \sum_{\vect x\in V} \sum_{\vect y\in {\cal N}(\vect x)}|w(\vect x,\vect y)(u(\vect x)-u(\vect y))|, \label{eqn:nabla-l1-n}
\end{gather}
and
\begin{equation} \label{eqn:nabla-l2-n}
\|\nabla {\vect u}\|^2 :=
\dfrac{1}{2}\vect{u}^\top L\vect{u}=\dfrac{1}{2}\sum_{\vect x\in V} \sum_{\vect y\in {\cal N}(\vect x)}w(\vect x,\vect y)(u(\vect x)-u(\vect y))^2,
\end{equation}
respectively, see e.g. \cite{LET12,YT18} for more detail.

%-------------------------------------------------------------------
\section{Proposed data classification method}\label{sec:method}
%-------------------------------------------------------------------
%--------------------------------------------
\subsection{Preliminary}
%--------------------------------------------
Given a point cloud $V$ containing $N$ points in $\mathbb{R}^M$. We aim to partition $V$ into $K$ classes $V_1, \cdots, V_K$ based on their similarities
(the points in the same class possess high similarity), with a set of training points $T=\{T_j\}_{j=1}^K \subset V$, $|T| = N_T$. Note that $T_j \subset V_j$ for $j = 1, \ldots, K$. In other words, we aim to assign the points in $V\setminus T$ certain labels between 1 to $K$ using the training set $T$ in which the labels of points are known, and the partition satisfies no vacuum and overlap constraint:
\begin{equation} \label{eqn:partition}
V=\bigcup_{j=1}^K V_j \quad {\rm  and  } \quad V_i\cap V_j = \emptyset, \quad \forall i\neq j, 1\le i, j \le K.
\end{equation}
In the rest of the paper, we denote the points in $V$ needed to be labeled as $S = V\setminus T$, and call $S$ the test set in $V$.

The constraint \eqref{eqn:partition} can be described by a binary matrix function
$U := ({\vect u}_1, \cdots, {\vect u}_K) \in \mathbb{R}^{N\times K}$ (also called partition matrix),
with ${\vect u}_j = (u_j(\vect x))_{\vect x\in V}^\top \in \mathbb{R}^N : V \rightarrow \{0, 1\}$ defined as
\begin{equation} \label{eqn:partition-label}
u_{j}(\vect x) :=
\begin{cases}
1, & \vect x \in V_j, \\
0, & {\rm otherwise},
\end{cases}
\quad
\forall \vect x\in V, j = 1, \ldots, K.
\end{equation}
Clearly, the above definition yields $\sum_{j=1}^K {u}_j(\vect x) = 1, \forall \vect x \in V$. The constraint \eqref{eqn:partition-label} is also known as the indicator constraint on the unit simplex. Since the binary representation in \eqref{eqn:partition-label} generally requires solving a non-convex model with NP-hard issue, a common strategy---the convex unit simplex---is considered as an alternative
\begin{equation} \label{eqn:partition-label-relax}
\sum_{j=1}^K {u}_{j} (\vect x) = 1, \quad\forall \vect x \in V, \quad {\rm s.t.} \quad {u}_{j}(\vect x) \in [0, 1], j = 1, \ldots, K.
\end{equation}
Note, importantly, that the convex constraint \eqref{eqn:partition-label-relax} can overcome the NP-hard issue and make some subproblems
convex, but generally the whole model can still be non-convex. Therefore, solving a model with constraint \eqref{eqn:partition}, \eqref{eqn:partition-label},
or \eqref{eqn:partition-label-relax} can be time consuming, see e.g.  \cite{LET12,YT18} for more detail.

If a result satisfying \eqref{eqn:partition-label-relax} is not completely binary, a common way to obtain an approximate binary solution satisfying \eqref{eqn:partition-label} is to select the binary function as the nearest vertex in the unit simplex by
the magnitude of the components, i.e.
\begin{equation} \label{eqn:partition-label-binary}
(u_1 (\vect x), \cdots, u_K (\vect x)) \mapsto \vect e_i, \quad {\rm where} \ i = \underset{j}{\rm argmax} \{u_{j} (\vect x)\}_{j=1}^K, \forall \vect x \in V.
\end{equation}
Here, $\vect e_i$ is the $K$-length unit normal vector which is 1 at the $i$-th component and 0 for all other components.

%--------------------------------------------
\subsection{Proposed method}
%--------------------------------------------
In this section, we present our novel two-stage method for data (e.g. point cloud) classification based on the SaT strategy which has been validated very effective in image segmentation. Our method can be summarized briefly as follows: first, a classification result is obtained as a warm initialization by using a classical, fast, but need
not be very accurate classification method such as SVM \cite{CV95}; then, a proposed two-stage iteration scheme is implemented until no change in the labels of the test points could be made between consecutive iterations. Specifically, at the first stage, we propose to minimize a novel convex model free of constraint (like those in \eqref{eqn:partition}, \eqref{eqn:partition-label} and \eqref{eqn:partition-label-relax}), to obtain a fuzzy partition, say $U$, while keeping the training labels unchanged; at the second stage, a binary result is obtained by just applying the binary rule in \eqref{eqn:partition-label-binary} directly on the fuzzy partition obtained at the first stage.
This binary result could be the final classification result for the original classification problem, or, if necessary, be set as a new initialization to search a better one in the same manner. In the following, we give the details of each step.

%Let $\hat{U} = (\hat{\vect u}_1, \cdots, \hat{\vect u}_K) \in \mathbb{R}^{N\times K}$ represent the obtained initialization on $V$, where $\hat{\vect u}_j = (\hat{u}_j(x))_{x\in V}^\top \in \mathbb{R}^N$ for $j = 1, \ldots, K$. Notice that since we already have the labels on the points in the training set $T$, let
%$\bar{U} = (\bar{\vect u}_1, \cdots, \bar{\vect u}_K) \in \mathbb{R}^{N_T\times K}$
%be the partition matrix on $T$ (note that $|T| = N_T$), where $\bar{\vect u}_j = (\bar{u}_j(x))_{x\in T}^\top \in \mathbb{R}^{N_T}$ for $j = 1, \ldots, K$. Now it suffices to assign labels to points in the test set $S$, we have
%\begin{equation}
%\hat{u}_j(x) = \bar{u}_j(x), \quad  \forall x\in T, j = 1, \ldots, K.
%\end{equation}
%Let $\hat{\vect u}_{S_j}$ represent the part of $\hat{\vect u}_{j}$ defined on the test set $S$, then we have
%\begin{equation} \label{eqn:decom-u-ini}
%\hat{\vect u}_{j} = (\hat{\vect u}_{S_j}^\top, \bar{\vect u}_j^\top)^\top, \quad   j = 1, \ldots, K.
%\end{equation}
%In the following of this paper, we use analogous notations for partition matrix $U = ({\vect u}_1, \cdots, {\vect u}_K)$, with
%\begin{equation} \label{eqn:decom-u}
%{\vect u}_j = ({\vect u}_{S_j}^\top, \bar{\vect u}_j^\top)^\top, j=1, \ldots, K,
%\end{equation}
%where ${\vect u}_{S_j}$
%and $\bar{\vect u}_j$ are the labeling vectors corresponding to the test set $S$ and training set $T$, respectively.

{\it Initialization.} Given a point cloud $V$ containing $N$ points in $\mathbb{R}^M$ and training set $T$ containing $N_T$ points with correct labels, we use SVM, which is a standard and fast clustering method as an example, to obtain the first clustering. Let the partition matrix be
$\hat U= (\hat{\vect u}_1, \cdots, \hat{\vect u}_K) \in \mathbb{R}^{N\times K}$, where $\hat{\vect u}_j = (\hat{u}_j(\vect x))_{\vect x\in V}^\top \in \mathbb{R}^N$ for $j = 1, \ldots, K$. One could acquire an initialization by any other
methods which have better performance than SVM. If no proper method is available
(e.g. the data set is too large), then an initialization generated
by setting labels to the test points randomly can be used as an alternative.

{\it Stage one.} Now we put forward our convex model to find a fuzzy partition $U$ with initialization $\hat{U} = (\hat{\vect u}_1, \cdots, \hat{\vect u}_K)$. It is
\begin{align} \label{eqn:model-proposed}
\underset{U}{\argmin} \sum_{j=1}^K\left\{\frac{\beta}{2}\|{\vect u}_j-\hat{\vect u}_j\|_2^2 + \frac{\alpha}{2}{\vect u}_j^{\top} L{\vect u}_j + \|\nabla {\vect u}_j\|_1 \right\},
\end{align}
where the first term is the data fidelity term constraining the fuzzy partition not far away from the initialization; the second term is related to $\|\nabla {\vect u}\|_2^2$ with graph Laplacian $L$; the last term is the total variation constructed on the graph; and $\alpha, \beta > 0$ are regularization parameters. Specifically, the second term in \eqref{eqn:model-proposed} is used to impose smooth features on the labels of the points, and the last term is used to force the points with similar information to group together.

We emphasize that we already have the labels on the points in the training set $T$, with $\bar U= (\bar{\vect u}_1, \cdots, \bar{\vect u}_K) \in \mathbb{R}^{N_T\times K}$ being the partition matrix on $T$, where $\bar{\vect u}_j = (\bar{u}_j(\vect x))_{\vect x\in T}^\top \in \mathbb{R}^{N_T}$ for $j = 1, \ldots, K$. Therefore, we only assign labels to points in the test set $S$, i.e. we have
\begin{equation}
\hat{u}_j(\vect x) = \bar{u}_j(\vect x), \quad  \forall \vect x\in T, j = 1, \ldots, K.
\end{equation}
Let $\hat{\vect u}_{S_j}$ represent the part of $\hat{\vect u}_{j}$ defined on the test set $S$, then we have
\begin{equation} \label{eqn:decom-u-ini}
\hat{\vect u}_{j} = (\hat{\vect u}_{S_j}^\top, \bar{\vect u}_j^\top)^\top, \quad   j = 1, \ldots, K.
\end{equation}
We use analogous notations for the partition matrix $U = ({\vect u}_1, \cdots, {\vect u}_K)$, with
\begin{equation} \label{eqn:decom-u}
{\vect u}_j = ({\vect u}_{S_j}^\top, \bar{\vect u}_j^\top)^\top, j=1, \ldots, K.
\end{equation}
In Section \ref{sec:alg}, \eqref{eqn:decom-u-ini} and \eqref{eqn:decom-u} are going to be used to derive an efficient algorithm to solve \eqref{eqn:model-proposed}.

The following Theorem \ref{thm:unique} proves that our proposed model \eqref{eqn:model-proposed} has a unique solution.
\begin{theorem} \label{thm:unique}
Given $\hat{U}\in \mathbb{R}^{N\times K}$ and $\alpha, \beta > 0$,
the proposed model \eqref{eqn:model-proposed} has a unique solution ${U}\in \mathbb{R}^{N\times K}$.
\end{theorem}
\begin{proof}
According to \cite[Chapter 9]{Boyd04}, a strongly convex function has a unique minimum. The conclusion follows easily from the strong convexity of model \eqref{eqn:model-proposed}.
\end{proof}

Many algorithms can be used to solve model \eqref{eqn:model-proposed} efficiently due to the convexity of the model without constraint. For example, the split-Bregman algorithm \cite{GO09}, which is specifically devised for $\ell_1$ regularized problems; the primal-dual algorithm \cite{CP11}, which is designed to solve general saddle point problems; and the
powerful alternative, ADMM algorithm \cite{BPCPE11}.
In particular, model \eqref{eqn:model-proposed} actually contains $K$ independent sub-minimization problems,
where each corresponds to a labeling function ${\vect u}_j$, and therefore the parallelism strategy is ideal to apply. This is an important advantage of our method for large data sets.
The algorithm aspects to solve our proposed convex model \eqref{eqn:model-proposed} are detailed in Section \ref{sec:alg}.

{\it Stage two.} This stage is to project the fuzzy partition result $U$ obtained at stage one to a binary partition. Here, formula \eqref{eqn:partition-label-binary} is applied to the fuzzy partition $U$ to generate a binary partition, which naturally satisfies no vacuum and overlap constraint \eqref{eqn:partition}.
We remark that compared to the computation time at stage one, the time at stage two is negligible.

Normally, the classification work is complete after we obtain a binary partition matrix at stage two. However, since the way of obtaining an initialization in our scheme is open,
and the quality of the initialization could be poor, we suggest going back to stage one with the latest obtained partition as a new initialization and repeat the above two stages until no more change in the partition matrix is observed. More precisely, we set $U$ as $\hat{U}$ and repeat Stages 1 and 2 again to obtain a new $U$. Then the final classification result is the converged stationary partition matrix, say $U^*$.
Moreover, to accelerate the convergence speed, we update $\beta$ in \eqref{eqn:model-proposed} by a factor of $2$ if we are to repeat the stages. This will obviously enforce the closeness between two consecutive clustering
results during iterations, which will ensure the algorithm converges fast.
We stop the algorithm when no changes are observed in the clustering result compared to the previous one.
We remark that a few iterations ($\approx$ 10) are generally enough in practice, see the experimental results in Section \ref{sec:numerics} for more detail.

Note, importantly, that our classification method here is totally different from other variational methods like \cite{LET12,YT18} which need to minimize
variational models with constraint like \eqref{eqn:partition}, \eqref{eqn:partition-label}, \eqref{eqn:partition-label-relax}, or other kinds of constraints
(e.g. minimum and maximum number of points imposed on individual classes $V_i$).
Even though our proposed model \eqref{eqn:model-proposed} has no constraint,
the final classification result of our method naturally satisfies no vacuum and overlap constraint \eqref{eqn:partition}. Therefore, our method is much easier to solve for each iteration. Our proposed method, namely SaT
(smoothing and thresholding) method for high-dimensional data classification,
is summarized in Algorithm \ref{alg:pcs}.
%-----------------
\begin{algorithm}[h]
 \caption{SaT method for high-dimensional data classification}
 \label{alg:pcs}

{\bf Initialization:} Generate initialization $\hat{U}$ by e.g. SVM method. \\
{\bf Output:} Binary partition $U^*$.

\vspace{0.05in}

{\bf For} $l = 0,1,\ldots,$ until the stopping criterion reached (e.g. $\|U^{(l)} - U^{(l+1)}\| = 0$)

\hspace*{0.2in}{\it Stage one:}	Compute fuzzy partition ${U}$ by solving model 	\eqref{eqn:model-proposed}.

\hspace*{0.2in}{\it Stage two:} Compute binary partition ${U}^{(l+1)}$ by using formula \eqref{eqn:partition-label-binary} on ${U}$.

\hspace*{0.17in} Set $\hat{U} = U^{(l+1)}$ and $\beta = 2\beta$.

{\bf Endfor} \\
Set $U^* = U^{(l+1)}$.

\end{algorithm}
%-----------------

%---------------------------------------------------------------------------------------
\section{Algorithm aspects} \label{sec:alg}
%---------------------------------------------------------------------------------------
In this section, we present an algorithm to solve the proposed convex model \eqref{eqn:model-proposed} based on the primal-dual algorithm \cite{CP11}
which is briefly recalled below.

%--------------------------------------------
\subsection{Primal-dual algorithm}
%--------------------------------------------
Let ${X}_i$ be a finite dimensional vector space  {equipped} with a proper inner product $\langle\cdot,\cdot\rangle_{X_i}$ and norm $\|\cdot\|_{X_i}$, $i = 1, 2$. Let map ${\cal K}:X_1\rightarrow X_2$ be a bounded linear operator.
The primal-dual algorithm is, generally speaking, to solve the following saddle-point problem
\begin{equation} \label{eqn:cp-problem}
\min_{{\vect x}\in X_1}\max_{\tilde{\vect x} \in X_2}
\Big\{
\langle {\cal K}{\vect x}, \tilde{\vect x}\rangle + {\cal G}(\vect x) - {\cal F}^*(\tilde{\vect x})
\Big\},
\end{equation}
where ${\cal G}:X_1\rightarrow[0,+\infty]$, ${\cal F}:X_2\rightarrow[0,+\infty]$ are proper, convex,
lower-semicontinuous functions, and ${\cal F}^*$ represents the convex conjugate of ${\cal F}$.
Given proper initializations, the primal-dual algorithm to solve \eqref{eqn:cp-problem} can be summarized in the following  iterative way of updating the primal and dual variables
\begin{align}
\tilde{\vect x}^{(l+1)} & =(I+\sigma\partial {\cal F}^*)^{-1}(\tilde{\vect x}^{(l)}+\sigma {\cal K} {\vect z}^{(l)}), \\
{\vect x}^{(l+1)} & = (I+\tau\partial {\cal G})^{-1}({\vect x}^{(l)}-\tau {\cal K}^*\tilde{\vect x}^{(l+1)}), \\
{\vect z}^{(l+1)} & = {\vect x}^{(l+1)}+\theta({\vect x}^{(l+1)} - {\vect x}^{(l)}),
\end{align}
where $\theta\in[0,1], \tau, \sigma>0$ are algorithm parameters.

%--------------------------------------------
\subsection{Algorithm to solve our proposed model}
%--------------------------------------------
We first define some useful notations which will be used to present our algorithm.

%--------------------------------------------
\subsubsection{Preliminary}
%--------------------------------------------
For ease of explanation, in the following, when we say $(i, j) \in E$, the $i$ and $j$ represent the $i$-th and $j$-th vertices in $E$,
respectively. Let
\begin{equation}
E^\prime = \big\{(i, j) \ | \ i < j, \forall (i, j) \in E \big\}.
\end{equation}
The graph laplacian $L = D - W \in \mathbb{R}^{N\times N}$ can be decomposed as
\begin{equation}\label{decompL}
L = \sum_{(i, j) \in E^\prime}L_{ij},
\end{equation}
where
\begin{equation}
 L_{ij}=\bordermatrix{ & & i & & j &\cr
 & & \vdots & & \vdots &  \cr
i & \cdots&w(i,j)&\cdots&-w(i,j)&\cdots\cr
 & &\vdots& &\vdots& \cr
j &\cdots&-w(i,j)&\cdots&w(i,j)&\cdots\cr
 & &\vdots& &\vdots& }
 \in \mathbb{R}^{N\times N}
\end{equation}
is a matrix with only four nonzero entries which locate at positions $(i, i), (i, j), (j, i)$ and $(j, j)$.
Let $E^\prime = E^\prime_a \cup E^\prime_b \cup E^\prime_c$,
where
\begin{align}
E^\prime_a &=  \big\{(i, j) \ | \ i, j\in S, \forall (i, j) \in E^\prime \big\}, \\
E^\prime_b &= \big\{(i, j) \ | \ i, j\in T, \forall (i, j) \in E^\prime \big\}, \\
E^\prime_c &= E^\prime \setminus (E^\prime_a \cup E^\prime_b).
\end{align}
Then the decomposition $L$ in \eqref{decompL} can be rewritten as
\begin{equation}\label{eqn:L_s}
L = \sum_{(i,j)\in E^\prime_a}L_{ij}+\sum_{(i,j)\in E^\prime_b}L_{ij}+\sum_{(i,j)\in E^\prime_c}L_{ij}.
\end{equation}
Note that, the terms $\sum_{(i,j)\in E^\prime_a}L_{ij}$ and $\sum_{(i,j)\in E^\prime_b}L_{ij}$ only have nonzero entries which are associated to the test set $S$
and the training set $T$, respectively. Let
\begin{equation}\label{eqn:L_S}
\sum_{(i,j)\in E^\prime_a}L_{ij} =
\begin{pmatrix}
L_S & 0 \\
\\
0 & 0
\end{pmatrix},
\sum_{(i,j)\in E^\prime_b}L_{ij} =
\begin{pmatrix}
0 & 0 \\
\\
0 & \bar L
\end{pmatrix},
\sum_{(i,j)\in E^\prime_c}L_{ij} =
\begin{pmatrix}
L_1 & L_3\\
\\
L_3^\top & L_2
\end{pmatrix},
\end{equation}
where $L_S, L_1 \in \mathbb{R}^{(N-N_T)\times (N-N_T)}$ are related to the test set $S$, $\bar{L}, L_2 \in \mathbb{R}^{N_T\times N_T}$ are related to the training set $T$, and $L_3 \in \mathbb{R}^{(N-N_T)\times N_T}$.
Then we have
\begin{equation} \label{eqn:L-decom}
{L} =
\begin{pmatrix}
{L}_{S}+L_1 & L_3 \\
\\
L_3^\top & \bar{L}+L_2
\end{pmatrix}.
\end{equation}

According to \eqref{eqn:nabla-n}, the gradient operator $\nabla$ can be regarded as a linear transformation between $\mathbb{R}^N$ and $\mathbb{R}^{N\times(k-1)}$
(where $k = |{\cal N}(\vect x)|$). For a vector ${\vect u}_j = ({\vect u}_{S_j}^\top, \bar{\vect u}_j^\top)^\top$ defined in \eqref{eqn:decom-u}, let
\begin{equation}
{\cal A}_S({\vect u}_{S_j}) =
\nabla
\begin{pmatrix}
{\vect u}_{S_j} \\
{\bf 0}
\end{pmatrix}
\in \mathbb{R}^{N\times(k-1)}, \quad
H_j =
\nabla
\begin{pmatrix}
{\bf 0} \\
\bar{\vect u}_j
\end{pmatrix}
\in \mathbb{R}^{N\times(k-1)}.
\label{eqn:def_A}
\end{equation}
Clearly, ${\cal A}_S: \mathbb{R}^{N-N_T}\rightarrow\mathbb{R}^{N\times(k-1)}$ is an operator corresponding to the test set $S$, and $H_j$ is the gradient matrix corresponding to the training set $T$ which is fixed since $\bar{\vect u}_j$ is fixed. Then, we have
\begin{equation} \label{eqn:gradient-decom}
\nabla {\vect u}_j = \nabla
\begin{pmatrix}
{\vect u}_{S_j} \\
\bar{\vect u}_j
\end{pmatrix}
=
\nabla
\begin{pmatrix}
{\vect u}_{S_j} \\
{\bf 0}
\end{pmatrix}
+
\nabla
\begin{pmatrix}
{\bf 0} \\
\bar{\vect u}_j
\end{pmatrix}
= {\cal A}_S({\vect u}_{S_j}) + H_j.
\end{equation}

%--------------------------------------------
\subsubsection{Algorithm}
%--------------------------------------------
Substituting the decomposition of $L$ in \eqref{eqn:L-decom} and $\nabla$ in \eqref{eqn:gradient-decom}, $\hat{\vect u}_j$ in \eqref{eqn:decom-u-ini} and ${\vect u}_j$ in \eqref{eqn:decom-u} into
the proposed minimization model \eqref{eqn:model-proposed} yields
\begin{equation}\label{eqn:model-proposed-rewrite}
\underset{\{{\vect u}_{S_j} \}_{j=1}^K}{\argmin}
\sum_{j=1}^K
\left\{ \frac{\beta}{2}\|\hat{\vect u}_{S_j}-{\vect u}_{S_j}\|_2^2 + \frac{\alpha}{2}{\vect u}_{S_j}^{\top}{L}_S {\vect u}_{S_j}+\alpha {\vect u}_{S_j}^{\top}{L}_3 \bar{\vect u}_{j} + \|{\cal A}_S({\vect u}_{S_j}) + H_j\|_1 \right\}.
\end{equation}
Note, obviously, that solving the above model \eqref{eqn:model-proposed-rewrite} is equivalent to solving $K$ sub-minimization problems corresponding to each ${\vect u}_{S_j}$, $j = 1, \ldots, K$, which means that our proposed model inherently benefits from the parallelism computation.

For $1\leq j\leq K$, let
\begin{align}
{\cal G}_j({\vect u}_{S_j}) & = \frac{\beta}{2}\|\hat{\vect u}_{S_j}-\vect {\vect u}_{S_j}\|_2^2 + \frac{\alpha}{2}\vect {\vect u}_{S_j}^{\top}{L}_S \vect {\vect u}_{S_j}+\alpha \vect {\vect u}_{S_j}^{\top}{L}_3 \bar{\vect u}_{j}, \label{eqn:G} \\
{\cal F}_j(\tilde{\vect x}) & =\|\tilde{\vect x} + H_{j}\|_1. \label{eqn:F}
\end{align}
Using the definition of the $\ell_1$-norm given in \eqref{eqn:nabla-l1-n},
the conjugate of ${\cal F}_j$, ${\cal F}_j^*$, can then be calculated as
\begin{align} \label{eqn:F-con}
\begin{split}
{\cal F}_j^*(\vect p)&=\sup_{\tilde{\vect x} \in \mathbb{R}^{N\times(k-1)}}\ \langle \tilde{\vect x},\vect p\rangle - \|\tilde{\vect x} + H_j\|_1\\
&= -\langle \vect p, H_j \rangle + \chi_P(\vect p),
\end{split}
\end{align}
where $P=\{\vect p\in \mathbb{R}^{N\times(k-1)}: \|\vect p\|_{\infty} \le 1\}$, and $\chi_P(\vect p)$ is the characteristic function of set $P$ with value 0 if
${\vect p} \in P$, otherwise $+\infty$.

Using the primal-dual formulation in \eqref{eqn:cp-problem} with the definitions of ${\cal G}_j$ and ${\cal F}_j^*$ respectively given in \eqref{eqn:G} and \eqref{eqn:F-con}, then the minimization problem \eqref{eqn:model-proposed-rewrite} corresponding to each ${\vect u}_{S_j}$ can be reformulated as
\begin{equation} \label{eqn:model-proposed-rewrite-pd}
\underset{{\vect u}_{S_j}}\argmin\max_{\vect p} \Big\{ \langle {\cal A}_S({\vect u}_{S_j}), \vect p\rangle + {\cal G}_j({\vect u}_{S})+\langle \vect p, \vect h_j \rangle -\chi_P(\vect p) \Big\}.
\end{equation}
To apply the primal-dual method, it remains to compute $(I+\sigma\partial {\cal F}_j^*)^{-1}$ and \mbox{$(I+\tau\partial {\cal G}_j)^{-1}$}.
Firstly, for $\forall \tilde{\vect x} \in \mathbb{R}^{N\times(k-1)}$, we have
\begin{align} \label{eqn:F-sol}
\begin{split}
(I+\sigma\partial {\cal F}^*)^{-1}(\tilde{\vect x})
  &= \underset{\vect p \in \mathbb{R}^{N\times(k-1)}}{\argmin}\ {\cal F}_j^*(\vect p) + \frac{1}{2\sigma}\|\vect p - \tilde{\vect x}\|_2^2\\
  &= \underset{\vect p \in \mathbb{R}^{N\times(k-1)}}{\argmin}\ \chi_P(\vect p)+\frac{1}{2\sigma}\|\vect p-\tilde{\vect x}\|_2^2-\langle\vect p, H_j \rangle\\
  &= \underset{\vect p \in \mathbb{R}^{N\times(k-1)}}{\argmin}\ \chi_P(\vect p)+\frac{1}{2\sigma}\|\vect p-\tilde{\vect x}-\sigma H_j\|_2^2\\
  &= {\iota}_P(\tilde{\vect x}+\sigma H_j),
\end{split}
\end{align}
where the operator ${\iota}_P(\cdot)$ is the pointwise projection operator onto the set $P$, i.e., $\forall p \in \mathbb{R}$,
\begin{equation}
{\iota}_P(p) =
\begin{cases}
1, & |p| > 1 \\
p, & {\rm otherwise}.
\end{cases}
\end{equation}
Secondly, for $\forall {\vect x} \in \mathbb{R}^{N-N_{T}}$, we have
\begin{align} \label{eqn:G-sol}
(I+\tau\partial {\cal G}_j)^{-1}({\vect x})
  = \underset{{\vect u}_{S_j} \in \mathbb{R}^{N-N_{T}} }{\argmin}\ {\cal G}_j({\vect u}_{S_j}) + \frac{1}{2\tau}\|{\vect u}_{S_j} - {\vect x}\|_2^2.
\end{align}
Using the definition of ${\cal G}_j({\vect u}_{S_j})$ given in \eqref{eqn:G}, problem \eqref{eqn:G-sol} becomes solving the following linear system
\begin{equation}
(\alpha L_S+\beta I+\frac{1}{\tau}I) {\vect u}_{S_j} = \beta \hat{\vect u}_{S_j}+\frac{1}{\tau} {\vect x} - \alpha L_3 \bar{\vect u}_j.
\end{equation}
Since $(\alpha \bar L+\beta I+\frac{1}{\tau}I)$ is positive definite, the above linear system can be solved efficiently by e.g. conjugate gradient method \cite{CN96}.

Finally, by exploiting the strong convexity of ${\cal G}_j,\forall 1\leq j\leq K$, which is shown in the next lemma, \cite{CP11} suggests that we could adaptively modify $\sigma,\tau$ to accelerate the convergence or the primal-dual method.
\begin{Lemma}
	The functions ${\cal G}_j,\forall 1\leq j\leq K$ are strongly convex with parameter $\beta$.
\end{Lemma}
\begin{proof}
	For simplicity, we omit the subscript $j$ and $S_j$ in the following proof. First, by \eqref{eqn:L_S}, $L_S$ is semi-positive definite. Therefore, $(\frac{\alpha}{2}\vect {\vect u}^{\top}{L}_S \vect {\vect u}+\alpha \vect {\vect u}^{\top}{L}_3 \bar{\vect u})$ is convex.	Now the strong convexity of $\cal G$ follows from the fact that the remaining term in \eqref{eqn:G}, which is $\frac{\beta}{2}\|\vect u-\hat{\vect u}\|_2^2$, is strongly convex with parameter $\beta$.
\end{proof}

The algorithm solving our proposed classification model \eqref{eqn:model-proposed-rewrite} (i.e. model \eqref{eqn:model-proposed}) is summarized in Algorithm \ref{alg:t-rof}, and its convergence proof is given in Theorem \ref{thm:conv} below. For each sub-minimization problem, the relative error between two consecutive iterations and/or a given maximum iteration number can be used as stopping criteria to terminate the algorithm. Finally, we emphasize again that our method is quite suitable for parallelism since the $K$ sub-minimization problems are independent with each other and therefore can be computed in parallel.

%-----------------
\begin{algorithm}[h]
 \caption{Algorithm solving the proposed model \eqref{eqn:model-proposed-rewrite} (i.e. model \eqref{eqn:model-proposed}) }
 \label{alg:t-rof}

{\bf Initialization:} $\tilde{\vect x}^{(0)} \in \mathbb{R}^{N\times(k-1)}$, ${\vect x}^{(0)}, {\vect z}^{(0)} \in \mathbb{R}^{N-N_{T}}$, $\theta\in[0,1], \tau^{(0)}, \sigma^{(0)}>0$. \\
{\bf Output:} $\{{\vect u}_{S_j}\}_{j=1}^K$.
\vspace{0.05in}

{\bf For} $j=1,\cdots, K$ (parallelism strategy can be applied)

\hspace*{0.2in} {\bf For} $l=0, 1,\ldots,$ until the stopping criterion reached \\
\hspace*{0.4in} Let $\tilde{\vect x} = \tilde{\vect x}^{(l)}+\sigma^{(l)} {\cal A}_S {\vect z}^{(l)}$, and compute $\tilde{\vect x}^{(l+1)}  =(I+\sigma^{(l)}\partial {\cal F}^*)^{-1}(\tilde{\vect x})$ by \eqref{eqn:F-sol};

\hspace*{0.4in} Let ${\vect x} = {\vect x}^{(l)}-\tau^{(l)} {\cal A}_S^*\tilde{\vect x}^{(l+1)}$, and compute ${\vect x}^{(l+1)}  = (I+\tau^{(l)}\partial {\cal G})^{-1}({\vect x})$ by \eqref{eqn:G-sol};

\hspace*{0.4in} Let $\theta^{(l)}=1/\sqrt{1+\beta\tau^{(l)}}$, and set $\tau^{(l+1)}=\theta^{(l)}\tau^{(l)},\sigma^{(l+1)}=\sigma^{(l)}/\theta^{(l)}$;

\hspace*{0.4in} Compute ${\vect z}^{(l+1)}  = {\vect x}^{(l+1)}+\theta({\vect x}^{(l+1)} - {\vect x}^{(l)})$;

\hspace*{0.2in} {\bf Endfor}\\
\hspace*{0.2in} Set ${\vect u}_{S_j} = {\vect x}^{(l+1)}$. \\
{\bf Endfor}

\end{algorithm}
%-----------------

\begin{theorem} \label{thm:conv}
{Algorithm 2} converges if $\tau^{(0)}\sigma^{(0)}<\frac{1}{N^2(k-1)}$.
\end{theorem}
\begin{proof}
By Theorem 2 in \cite{CP11}, {Algorithm 2} converges as long as $\|{\cal A}_S\|_2^2<\frac{1}{\tau^{(0)}\sigma^{(0)}}$. Therefore it suffices to find a suitable upper bound for $\|{\cal A}_S\|_2$. By our implementation in \eqref{eqn:def_A} and since the weight functions (\cref{eqn:radial_basis,eqn:z-m-p_weight,eqn:cos_weight}) take value between $[-1,1]$, each entry in ${\cal A}_S$ is between $[-1,1]$. Therefore, the 1-norm and $\infty$-norm of ${\cal A}_S$ can be easily estimated as
$$\|{\cal A}_S\|_1=\max_{1\leq j\leq N-N_T}\sum_{i=1}^{N(k-1)}|({\cal A}_S)_{ij}|\leq N(k-1)$$
and
$$\|{\cal A}_S\|_\infty=\max_{1\leq i\leq N(k-1)}\sum_{j=1}^{N-N_T}|({\cal A}_S)_{ij}|\leq N-N_T.$$
Now, we have
$$\|{\cal A}_S\|_2\leq \sqrt{\|{\cal A}_S\|_1\|{\cal A}_S\|_\infty}\leq N\sqrt{k-1}.$$
Therefore, we conclude that, the algorithm converges as long as we choose $\tau^{(0)},\sigma^{(0)} > 0$, such that $\tau^{(0)}\sigma^{(0)}<\frac{1}{N^2(k-1)}.$
\end{proof}

%-------------------------------------------------------------------
\section{Numerical results}\label{sec:numerics}
%-------------------------------------------------------------------
In this section, we evaluate the performance of our proposed method on four benchmark data sets---including \textsc{Three Moon, COIL, Opt-Digits} and \textsc{MINST}---for semi-supervised learning. \textsc{Three Moon} is a synthetic data set which has been used frequently e.g. \cite{GMBFP14,LET12,YT18}. The \textsc{COIL}, \textsc{Opt-Digits}, and \textsc{MNIST} data set can be found from the supplementary material of \cite{benchmarkbook}, the UCI machine learning repository\footnote{http://archive.ics.uci.edu/ml/datasets.html}, and the MNIST Database
of Handwritten Digits\footnote{http://yann.lecun.com/exdb/mnist/},
respectively. The basic properties of these test data sets are shown in Table \ref{dataInfo}.

\begin{table}[htbp]
\caption{Basic properties of the test benchmark data sets. ``Dimension" means the length of every vector representing individual points in the given data sets.}
\begin{center}
\begin{tabular}{|c c c c|}
\hline
Data set & Number of classes & Dimension & Number of points\\
\hline
\textsc{Three Moon} & 3 & 100 & 1500\\
\textsc{COIL} & 6 & 241 & 1500 \\
\textsc{Opt-Digits} & 10 & 64 & 5620 \\
\textsc{MNIST} & 10 & 784 & 70000 \\
\hline
\end{tabular}
\label{dataInfo}
\end{center}
\end{table}

To implement our method, $k$-NN graphs are constructed for the test data sets, using the randomized kd-tree \cite{SH08} to find the nearest neighbors with Euclidean distance as the metric. The radial basis function \eqref{eqn:radial_basis} is used to compute the weight matrix $W$, except for the \textsc{MNIST} data set where the Zelnic-Manor and Perona weight function \eqref{eqn:z-m-p_weight} is used with eight closed neighbors.
The training samples $T$---samples with labels known---are selected randomly from each test data set. The classification accuracy is defined as the percentage of correctly labeled data points.

For our proposed method, the regularization parameter $\beta$ is fixed to $10^{-4}$ for \textsc{MNIST}, $10^{-5}$ for \textsc{COIL}, and $10^{-2}$ for \textsc{Three Moon, Opt-Digits}. In practice, one could choose the value of $\beta$ based on the accuracy of initialization. The better the initialized accuracy, the larger $\beta$ we could choose. The regularization parameter $\alpha$ is set to 1 for \textsc{Three Moon, Opt-Digits}, 0.4 for \textsc{MNIST}, and $10^{-2}$ for \textsc{COIL}. The accuracy of the proposed method can be improved further after fine-tuning the values of $\alpha$ and $\beta$ for individual test data sets.
All the codes were run on a MacBook with 2.8 GHz processor and 16 GB RAM, and {\sc Matlab} 2017a.

%\red{Keep in mind that in Egil Bae's paper, they use different parameters for different test data, e.g. the number of neighbors and $c$. }

%--------------------------------------------
\subsection{Methods comparison}\label{sec:svm}
%--------------------------------------------
As mentioned in previous sections, we use SVM method \cite{CV95} to generate initializations for our proposed method. If it is not proper for a data set
(e.g. very slow due to the large size of the data set), we could just use an initialization generated by assigning clustering labels randomly.

The SVM is a technique aiming to find the best hyperplane that separates data points of one class from the others.
In practice, data may not be separable by a hyperplane. In that case, soft margin is used so that the hyperplane would
separate many data points if not all. It is also common to kernelize data points, and then find separating hyperplane in the transformed space.
The SVM method used in our experiments is trained with linear kernel.

We compare our proposed method with the state-of-the-art methods proposed recently, e.g.
CVM \cite{BM17}, GL \cite{GMBFP14}, MBO \cite{GMBFP14}, TVRF \cite{YT18}, LapRF \cite{YT18}, LapRLS \cite{SB11}, MP \cite{SB11}, and SQ-Loss-I \cite{benchmarkbook}.
The code TVRF was provided by the authors and the parameters used
in it were chosen by trial and error to give the best results.
%In particular, the code of method LapRF is for semi-supervised classification;
%the code of method CVM provided by the authors is for unsupervised classification and was written in {C} and run in {\sc Matlab}.
%
The classification accuracies of methods GL, MBO, LapRF, LapRLS, MP and SQ-Loss-I were taken from \cite{BM17,YT18},
in which methods CVM and TVRF were shown to be superior in most of the cases.

%--------------------------------------------
\subsection{Three moon data}
%--------------------------------------------
The synthetic {\sc Three Moon} data used here is constructed by following the way performed in \cite{BM17,YT18} exactly. We briefly repeat the procedure as follows. First, generate three half circles in $\mathbb{R}^2$---two half top unit circles and one half bottom circle with radius 1.5 which are centered at (0, 0), (3, 0) and (1.5, 0.4), respectively. Then 500 points are uniformly sampled from each half circle and embedded into $\mathbb{R}^{100}$ by appending zeros to the remaining dimensions. Finally, an i.i.d Gaussian noise with standard deviation 0.14 is added to each dimension of the data. An illustration of the first two dimensions of the {\sc Three Moon} data is shown in Fig. \ref{fig:threemoon} (a) where different colors are applied on each half circle. This is a three-class classification problem with the goal of classifying each half circle using a small number of supervised
points from each class. This classification problem is challenging due to the noise and the high dimensionality of all the points with high similarity in $\mathbb{R}^{98}$.

%%%
\begin{figure}[htbp]
\begin{center}
\begin{tabular}{c}
\includegraphics[trim={{.2\linewidth} {.2\linewidth} {.2\linewidth} {.0\linewidth}}, clip, width=0.5\linewidth, height = 0.4\linewidth]{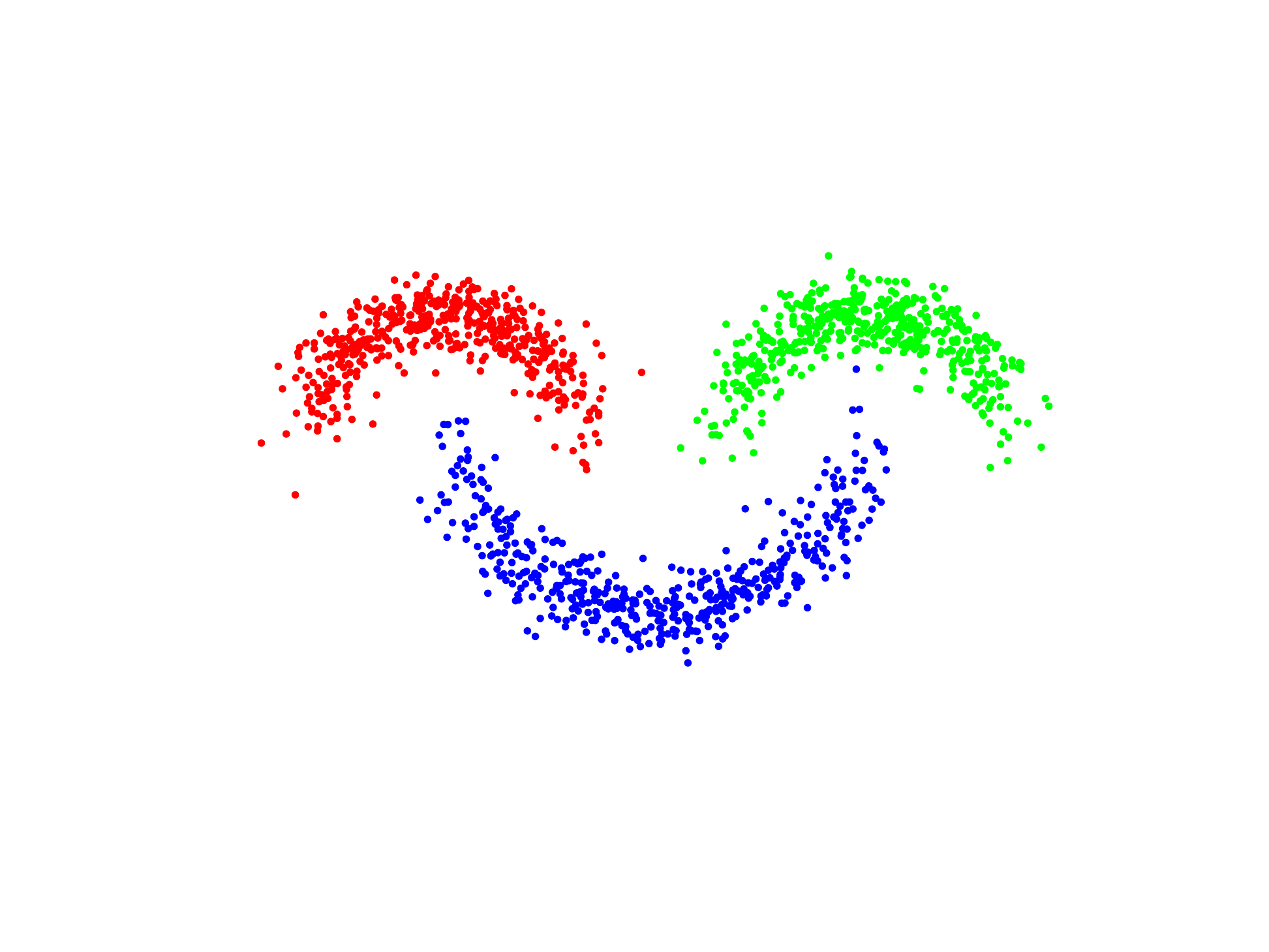} \\
(a) Ground truth
\end{tabular}
\begin{tabular}{cc}
\includegraphics[trim={{.2\linewidth} {.2\linewidth} {.2\linewidth} {.0\linewidth}}, clip, width=0.5\linewidth, height = 0.4\linewidth]{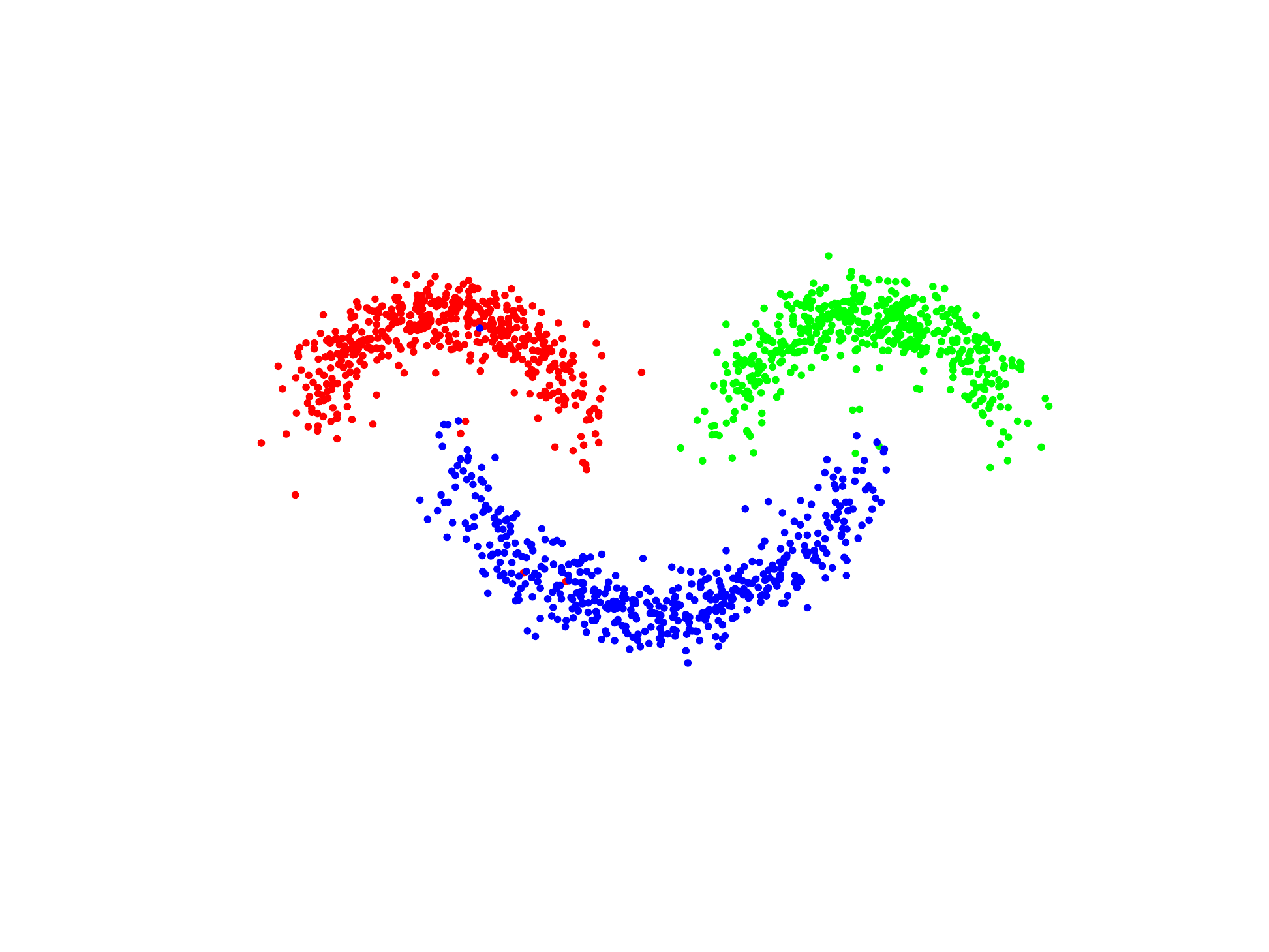} &
\includegraphics[trim={{.2\linewidth} {.2\linewidth} {.2\linewidth} {.0\linewidth}}, clip, width=0.5\linewidth, height = 0.4\linewidth]{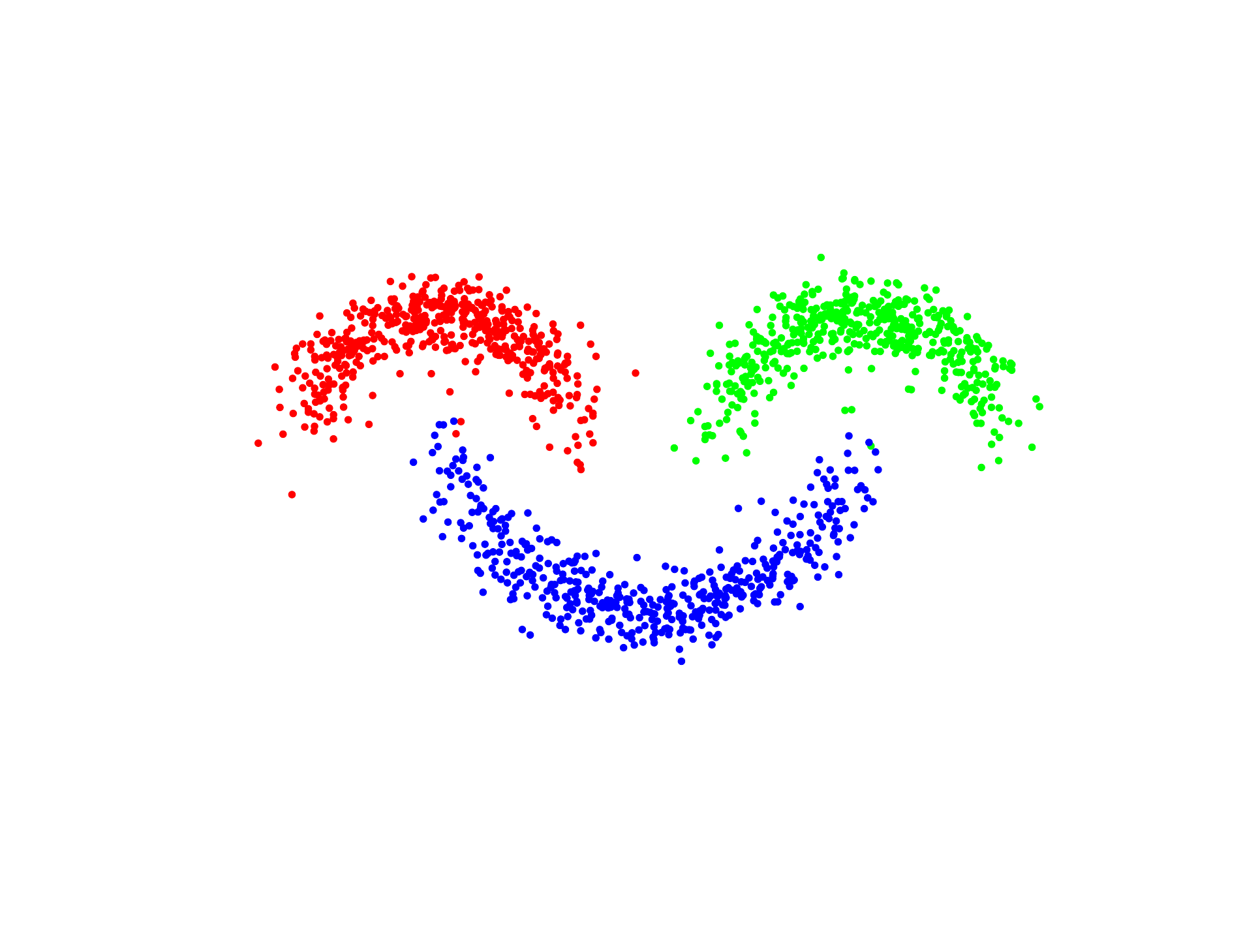} \\
(b) TVRF & (c) Our proposed method
\end{tabular}
\end{center}
\caption{Three-class classification for {\sc Three Moon} synthetic data. (a) Ground truth; (b) and (c) Results of method TVRF \cite{YT18} and our proposed method, respectively.  }
\label{fig:threemoon}
\end{figure}
%%%

A $k$-NN graph with $k=10$ is built for this data set, parameter $\sigma = 3$ is used in the Gaussian weight function, and the distance metric chosen is Euclidean metric for $\mathbb{R}^{100}$. We first test the methods using uniformly distributed supervised points, where a total number of 75 points are sampled uniformly from this data set as training points.

The accuracies of method TVRF and ours are obtained by running the methods ten times with randomly selected labeled samples, and taking the average of the accuracies. The accuracies of method CVM are copied from the original paper \cite{BM17}.
The accuracy comparison is reported in Table \ref{Table:threemoonaccuracy}, which shows that our proposed method gives the highest accuracy;
also, see Fig. \ref{fig:threemoon} for visual validation of the results between methods of TVRF and ours.
The average number of iterations taken for our proposed method is $3.8$. Fig. \ref{fig:accuracy} (a) gives the convergence history and partition accuracy of our proposed method corresponding to iteration steps, which clearly shows the accuracy increment during iterations (note that the accuracy at iteration 0 is the result of the initialization which is obtained by SVM method). Table \ref{Table:time} reports the comparison in terms of computation time, which indicates the superior performance of our proposed method in computation speed.

\begin{table}
\centering
      \begin{minipage}{0.4\linewidth}
      \centering
      \caption{Accuracy comparison for {\sc Three Moon} synthetic data set, with uniformly selected training points.}
          \label{Table:threemoonaccuracy}
      \begin{tabular}{|c c|}
	\hline
	Method & Accuracy(\%) \\
	\hline
	CVM & 98.7  \\
	GL & 98.4  \\
	MBO & 99.1   \\
	TVRF & 98.6 \\
	LapRF & 98.4   \\
	Proposed & \textbf{99.4} \\
	\hline
	\end{tabular}
     \end{minipage}
 \hspace{0.5cm}
     	\begin{minipage}{0.4\linewidth}
	\centering	
	\caption{Accuracy comparison for {\sc Three Moon} synthetic data set, with non-uniformly selected training points.}
    	\label{Table:unbalanced}
	\begin{tabular}{|c c|}
	\hline
	Method & Accuracy(\%)\\
	\hline
	TVRF & 97.8\\
	Proposed & \textbf{99.3}\\
	\hline
	\end{tabular}
	\end{minipage}
\end{table}

In the following, as a showcase, we test the methods using non-uniformly distributed supervised points, which is used to investigate the robustness of these methods on training points. In this case for the 75 training points, as an example, we respectively pick 5 points from the left and the bottom half circles, and pick the rest 65 points from the right half circle. This sampling is illustrated in Fig \ref{fig:unbalancedThreeMoon}.

\begin{figure}[htbp]
\begin{center}
\includegraphics[trim={{.2\linewidth} {.2\linewidth} {.2\linewidth} {.0\linewidth}}, clip, width=0.5\linewidth, height = 0.4\linewidth]{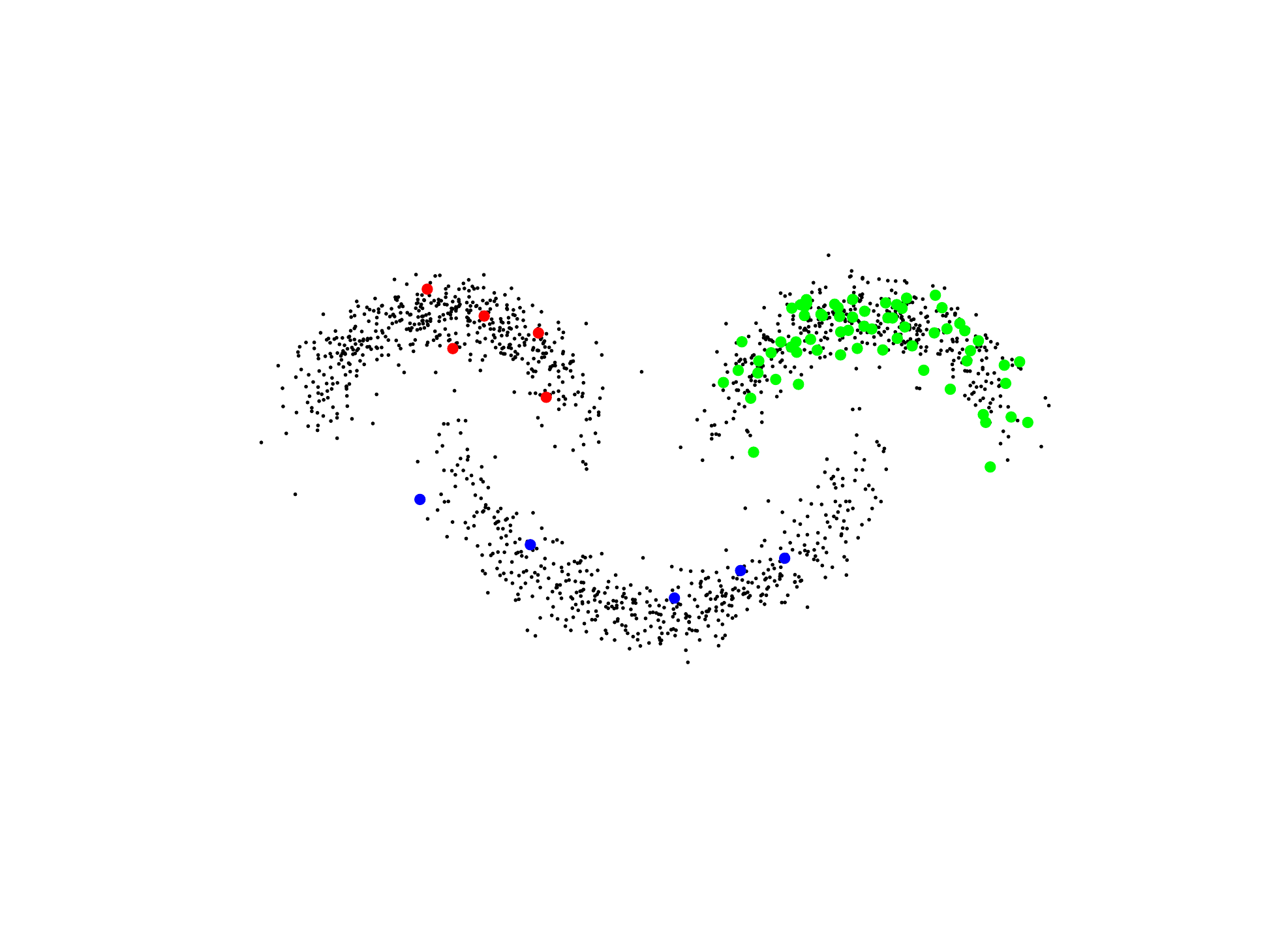}
\end{center}
\caption{Unbalanced sampling from {\sc Three Moon} data, where sampled points are highlighted with their corresponding labels. }
\label{fig:unbalancedThreeMoon}
\end{figure}

The accuracies of TVRF and our method are shown in Table \ref{Table:unbalanced}, from which we see clearly that the proposed method gives
much higher accuracy. In particular, compared to the results in Table \ref{Table:threemoonaccuracy} using training points chosen uniformly, while the accuracy of TVRF method decreases by 0.8 percent, we observe only a very small decrease in our proposed method.
This shows the robustness of our method with respect to the way that training points are selected. Note that in the case of training points chosen non-uniformly, the initialization obtained by SVM is poor, because of which more iterations are needed to converge for our method---average 12.0 iterations in 10 trials versus 3.3 iterations needed for the case of training points selected uniformly.

%--------------------------------------------
\subsection{COIL data}
%--------------------------------------------
The benchmark {\sc COIL} data comes from the Columbia object image library. It contains a set of color images of 100 different objects. These images, with size $128\times 128$ each, are taken from different angles in steps of 5 degrees, i.e., 72 ($=360/5$) images for each object. In the following, without loss of generality, we also call an image a point for ease of reference. The test data set here is constructed the same way as depicted in e.g. \cite{BM17,YT18} and is briefly described as follows.
First, the red channel of each image is down-sampled to $16\times 16$ pixels by averaging over blocks of $8\times 8$ pixels. Then, 24 out of the 100 objects are randomly selected, which amounts to 1728 $(= 24\times 360/5)$ images. After that, these 24 objects are partitioned into six classes with four objects---288 images ($= 4\times 72$)---in each class. Finally, after discarding 38 images randomly from each class, a data set of 1500 images where 250 images in each of the six classes are constructed.
To construct a graph, each image, which is a vector with length 241, is treated as a node on the graph,

For accuracy test, a $k$-NN graph with $k=4$ is built for this data set, parameter $\sigma = 250$ is used in the Gaussian weight function, and the distance metric chosen is Euclidean metric for $\mathbb{R}^{241}$. The training points, amount to 10\% of the points, are selected randomly from the data set. Again, we run the test methods 10 times and compare the average accuracy. The resulting accuracies are listed in Table \ref{Table:COILaccuracy}, showing that our method outperforms other methods. Moreover,
the average number of iterations of our method is 12.2. Fig. \ref{fig:accuracy} (b) gives the convergence history of our proposed method in
partition accuracy corresponding to iterations, which again shows an increasing trend in accuracy.

\begin{table}[htbp]
\centering
	\begin{minipage}{0.35\linewidth}
      \centering
	\caption{Accuracy comparison for {\sc COIL} data set, with uniformly selected training points.}
    \label{Table:COILaccuracy}
	\begin{tabular}{|c c|}
	\hline
	Method & Accuracy(\%)  \\
	\hline
	CVM & 93.3  \\
	TVRF & 92.5 \\
	LapRF & 87.7 \\
	GL & 91.2  \\
	MBO & 91.5   \\
	Proposed & \textbf{94.0} \\
	\hline
	\end{tabular}
    \end{minipage}
\hspace{0.5cm}
    \begin{minipage}{0.35\linewidth}
     \centering
     \caption{Accuracy comparison for {\sc MINST} data set, with uniformly selected training points.}
     \label{Table:MNISTaccuracy}
	\begin{tabular}{|c c|}
	\hline
	Method & Accuracy(\%)\\
	\hline
	CVM & \textbf{97.7}\\
	TVRF & 96.9\\
	LapRF & 96.9\\
	GL & 96.8\\
	MBO & 96.9\\
	Proposed & 97.5\\
	\hline
	\end{tabular}
    \end{minipage}
\end{table}

\subsection{MNIST data}
The MNIST data set consists of 70,000 images of handwritten digits 0--9, where each image has a size of $28\times 28$. Fig. \ref{fig:minst} shows some images of the ten digits from the data set. Each image is a node on a constructed graph. The objective is to classify the data set into 10 disjoint classes corresponding to different digits.
For accuracy test, a $k$-NN graph with $k=8$ is built for this data set, and Zelnik-Manor and Perona weight function \eqref{eqn:z-m-p_weight} is used to compute the weight matrix. The training 2500 (3.57\%) points (images) are selected randomly from the total 70,000 points.
The experimental results of the test methods are obtained by running them 10 times with randomly selected training set with a fixed number of points 2500, and the average accuracy is computed for comparison.
The accuracies of the test results are shown in Table \ref{Table:MNISTaccuracy}, which indicates that our method is comparable to or better than the state-of-the-art methods compared here. Table \ref{Table:time} shows the computation time comparison, from which we again see that our method is very competitive in computation speed. The convergence history of our proposed method in partition accuracy corresponding to iterations is given in Fig. \ref{fig:accuracy} (c), which also demonstrates a clear increasing trend in accuracy.

%%%
\begin{figure}[htbp]
\centerline{\includegraphics[trim={{.25\linewidth} {.5\linewidth} {.2\linewidth} {.4\linewidth}},clip,width=0.7\textwidth]{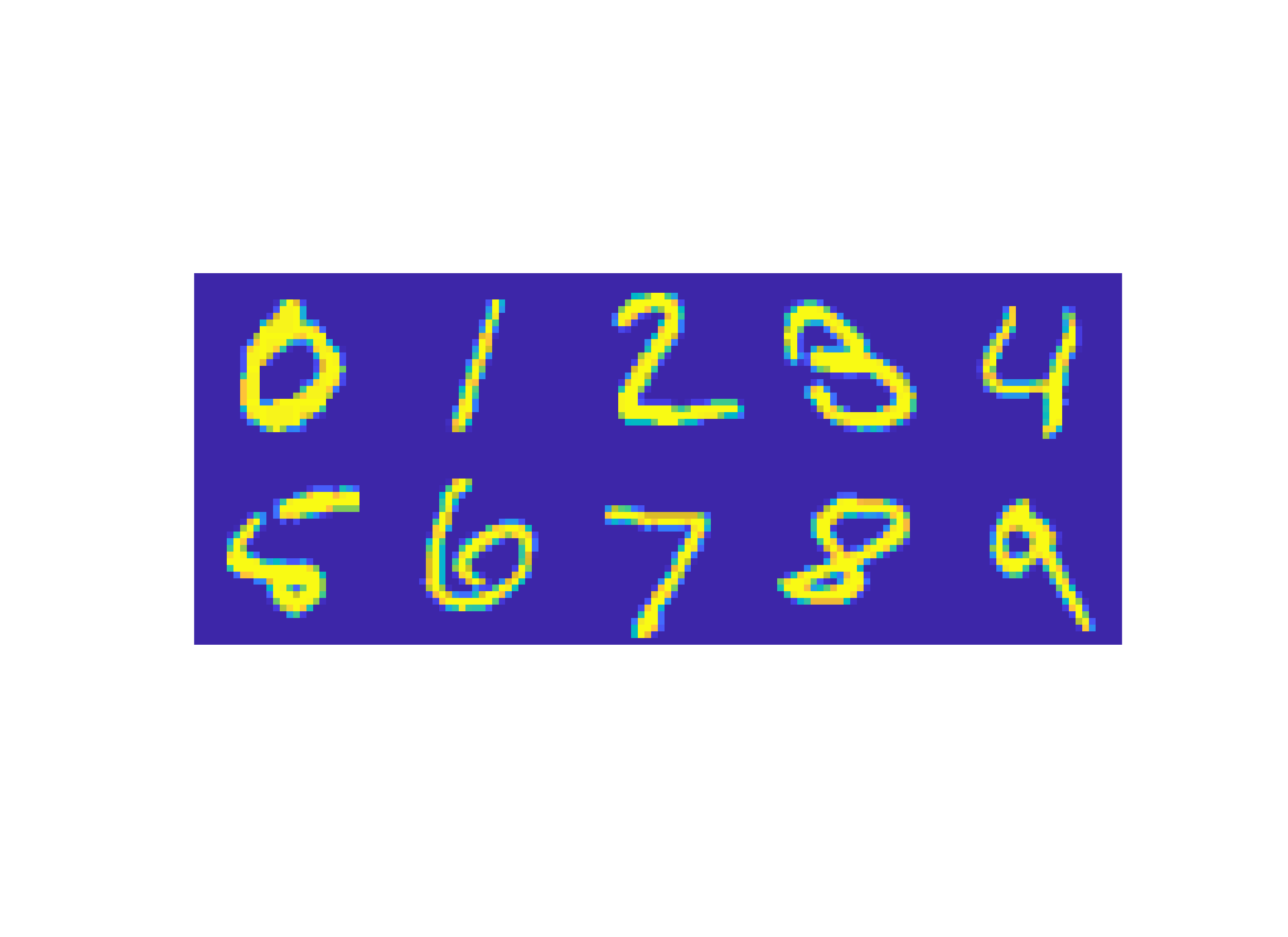}}
\caption{Examples of digits 0--9 from the MNIST data set.
}
\label{fig:minst}
\end{figure}
%%%

%--------------------------------------------
\subsection{Opt-Digits data}
%--------------------------------------------
The {\sc Opt-Digits} data set is constructed as follows.
It contains 5620 bitmaps of handwritten digits (i.e. 0--9). Each bitmap has the size of $32\times 32$ and is divided into non-overlapping blocks of $4\times 4$, and then the number of ``on" pixels are counted in each block. Therefore, each bitmap corresponds to a matrix of $8\times 8$ where each element is an integer in [0, 16]. The classification problem is to partition the data set into 10 classes.

For accuracy test, a $k$-NN graph with $k=8$ is built for this data set, parameter $\sigma = 30$ is used in the Gaussian weight function, and the distance metric chosen is Euclidean metric for $\mathbb{R}^{64}$. For the experiments on this data set, we generate three training sets respectively with the number of points 50, 100 and 150, which are all selected randomly.
All the methods are run 10 times for each training set
and the average accuracy is used for comparison. The quantitative results in accuracy are listed in Table \ref{Table:Opticsaccuracy}, from which we see that our proposed method is consistently better than the state-of-the-art methods compared for all the cases. We also observe the improvement of the accuracy of these methods w.r.t. the increasing number of points in the training set.
Finally, we show the convergence history of our proposed method in partition accuracy corresponding to iterations using the training set with 150 points in Fig. \ref{fig:accuracy} (d), which again clearly shows an increasing trend in accuracy.

\begin{table}[htbp]
\caption{Accuracy comparison for {\sc Opt-Digits} data set, with uniformly selected training points.}
\label{Table:Opticsaccuracy}
\begin{center}
\begin{tabular}{|c c c c|}
\hline
Sample rate & 0.89\%(50) & 1.78\%(100) & 2.67\%(150) \\
\hline
k-NN & 85.5 & 92.0 & 93.8  \\
SGT & 91.4 & 97.4 & 97.4  \\
LapRLS & 92.3 & 97.6 & 97.3   \\
SQ-Loss-I & 95.9 & 97.3 & 97.7  \\
MP & 94.7 & 97.0 & 97.1  \\
TVRF & 95.9 & 98.3 & 98.2   \\
LapRF & 94.1 & 97.7 & 98.1  \\
Proposed & \textbf{96.6} & \textbf{98.5} & \textbf{98.6} \\
\hline
\end{tabular}
\end{center}
\end{table}

%%%
\begin{figure}[htbp]
\begin{tabular}{cc}
\includegraphics[trim={{.0\linewidth} {.0\linewidth} {.15\linewidth} {.01\linewidth}}, clip, width=0.48\linewidth, height = 0.35\linewidth]{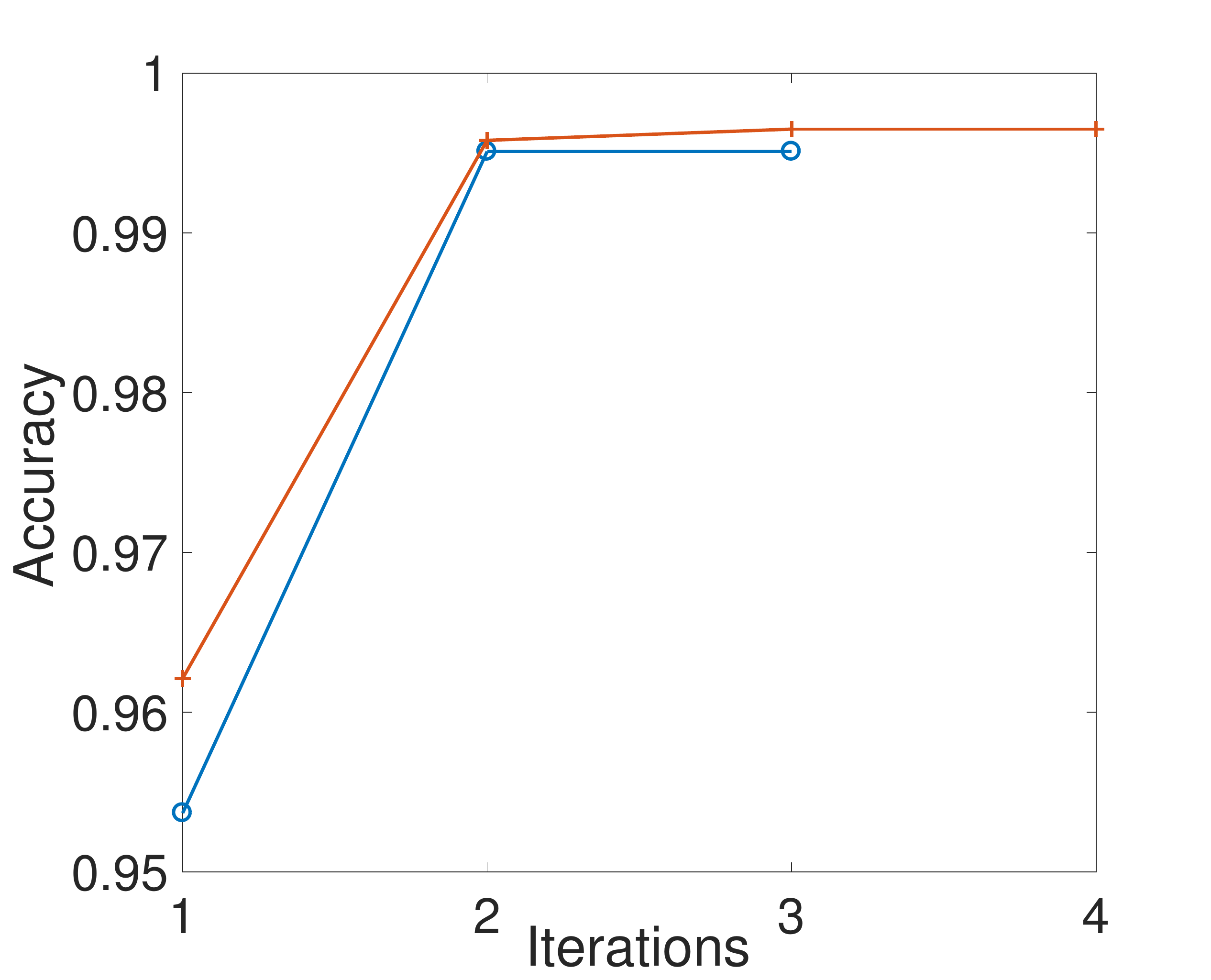} &
\includegraphics[trim={{.0\linewidth} {.0\linewidth} {.1\linewidth} {.01\linewidth}}, clip, width=0.48\linewidth, height = 0.35\linewidth]{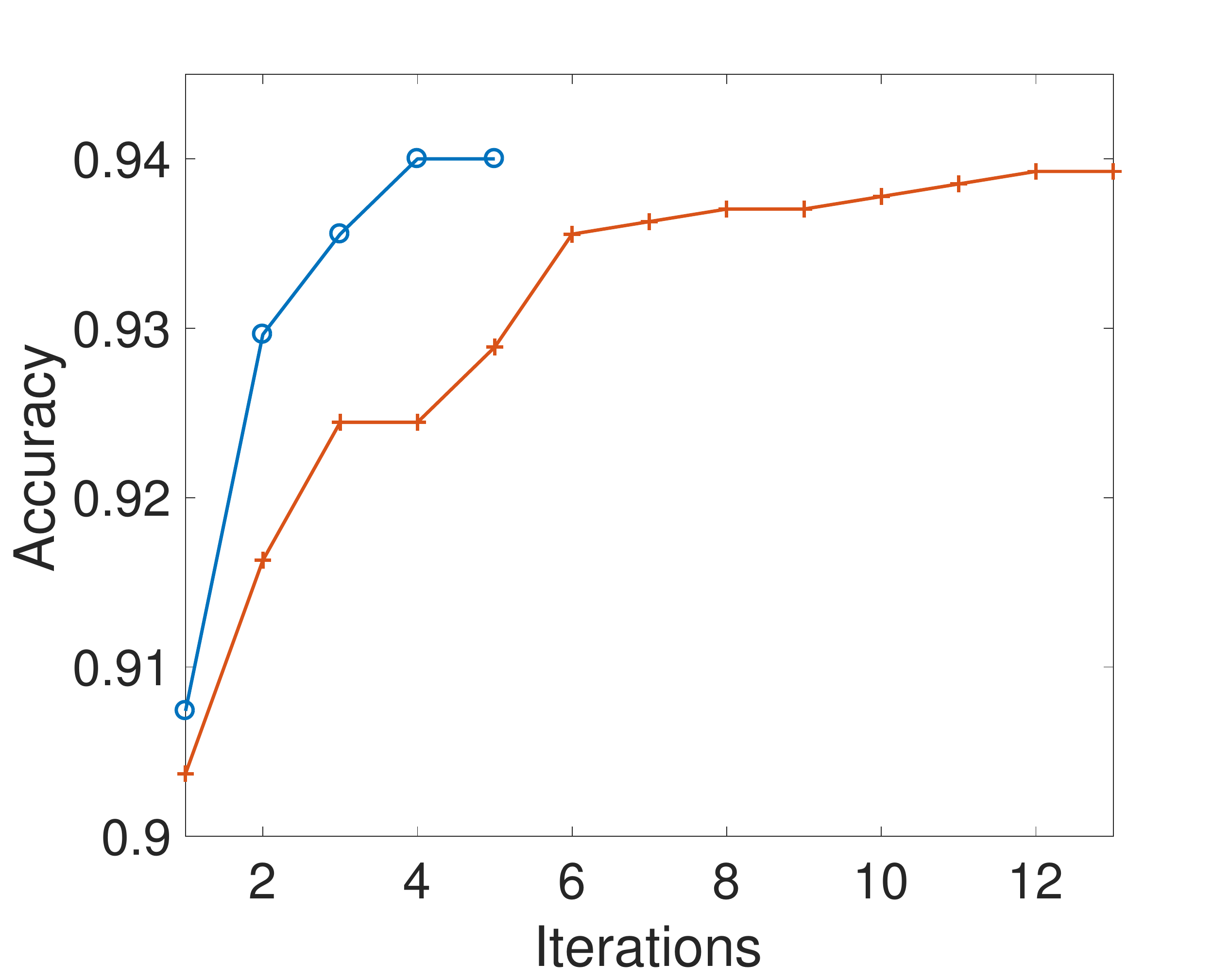}
\\
(a) \textsc{Three Moon} & (b) \textsc{COIL} \\
\includegraphics[trim={{.0\linewidth} {.0\linewidth} {.1\linewidth} {.01\linewidth}}, clip, width=0.48\linewidth, height = 0.35\linewidth]{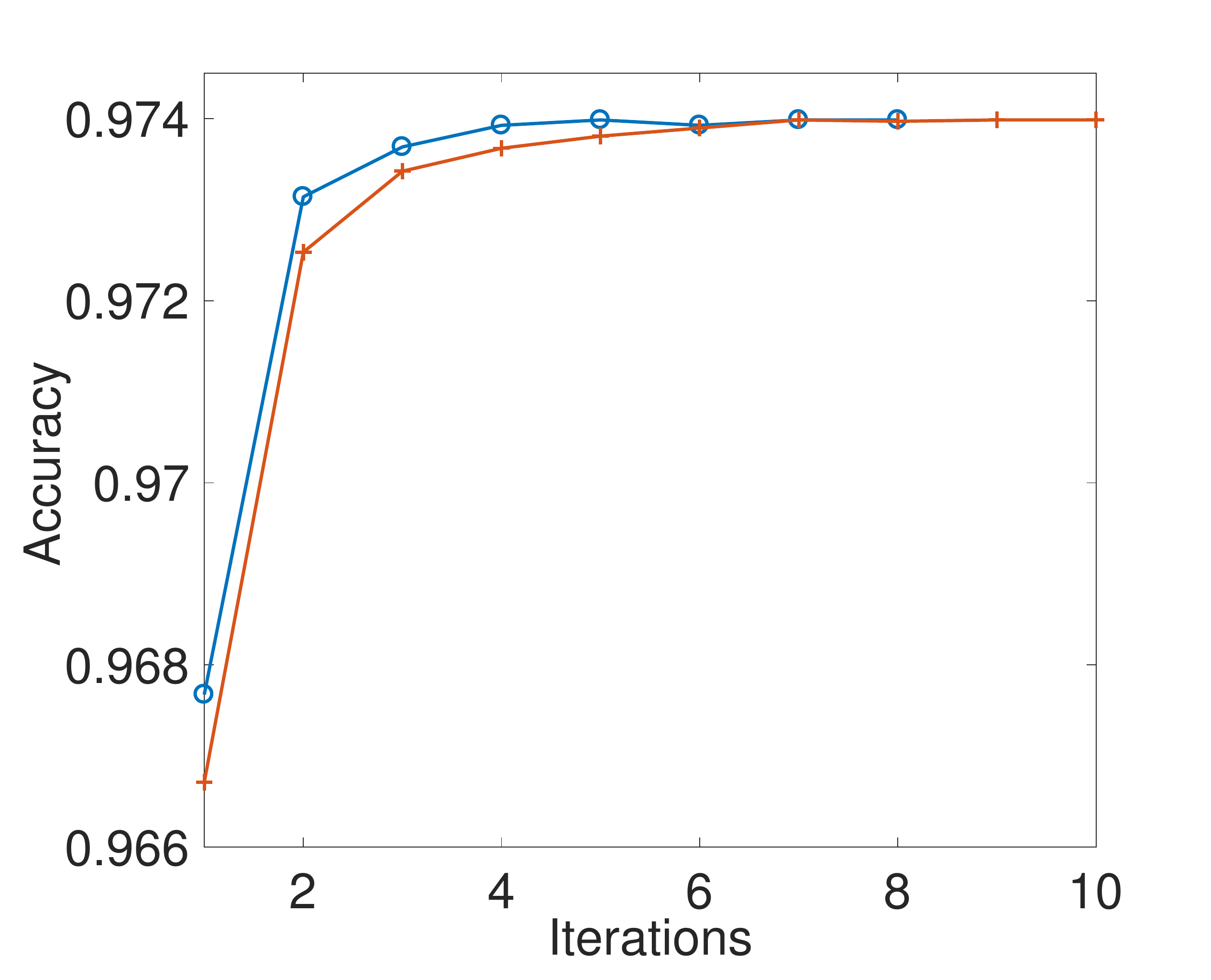} &
\includegraphics[trim={{.0\linewidth} {.0\linewidth} {.1\linewidth} {.01\linewidth}}, clip, width=0.48\linewidth, height = 0.35\linewidth]{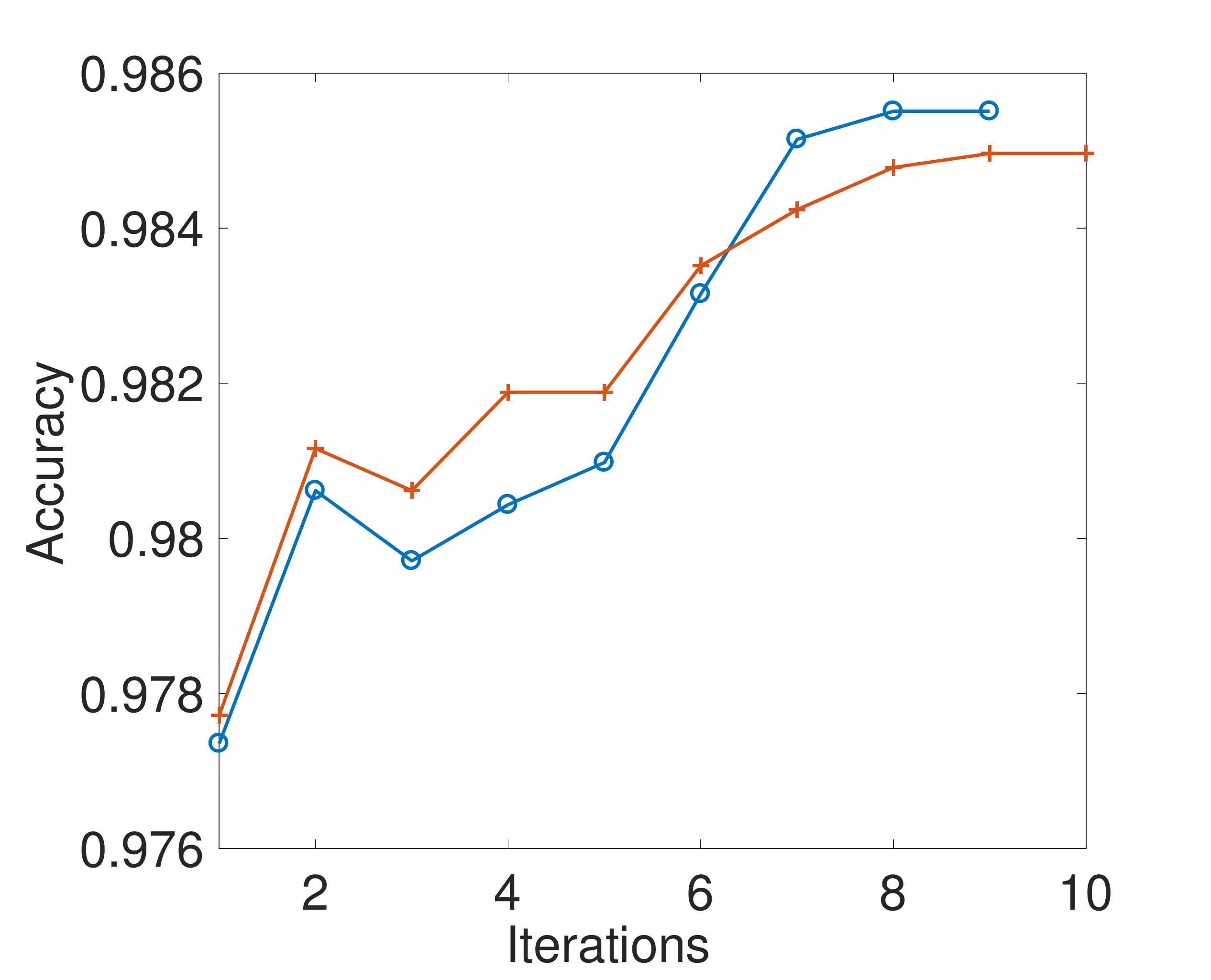}
\\
(c) \textsc{MINST} & (d) \textsc{Opt-Digits}
\end{tabular}
\caption{Accuracy convergence history of our proposed method corresponding to iteration steps for all the test data sets; training samples are uniformly selected in each class. Blue curves correspond to cases with the least number of iterations among the 10 trials and orange curves correspond to cases with the largest number of iterations among the 10 trials.}
\label{fig:accuracy}
\end{figure}
%%%

%%%%
%\begin{figure}[htbp]
%\centerline{\includegraphics[scale=0.4]{./figs/unbalanced_comparison.eps}}
%\caption{Accuracy convergence history of our proposed method corresponding to iteration numbers for test data \textsc{Three Moon}; training samples are non-uniformly selected in each class. Blue curve corresponds to the case with the least number of iterations and orange curve corresponds to the case with the largest number of iterations.
%}
%\label{fig:unbalancedHistory}
%\end{figure}
%%%

\begin{table}[htbp]
\caption{Computation time comparison in seconds. The value in the brackets of our method represents the average number of iterations of the 10 trials. (For Opt-Digits data, we select 100 sample points.)}
\begin{center}
\begin{tabular}{|c | l l l l |}
\hline
\multirow{2}{*}{Method} & \multicolumn{4}{c|}{Computation time in seconds}\\ \cline{2-5}
 & \textsc{Three Moon} & \textsc{COIL} & \textsc{MINST}  & \textsc{Opt-Digits}\\
\hline
\multicolumn{1}{|c|}{TVRF} & 0.71 & 0.65 & 66.00 &  3.42 \\
\multicolumn{1}{|c|}{Proposed} & \textbf{0.30} (3.3) & 0.76 (11.7) & 82.04 (9.4) & 4.45 (9.3) \\
\hline
\end{tabular}
\label{Table:time}
\end{center}
\end{table}

%--------------------------------------------
\subsection{Further discussion}
%--------------------------------------------
The above experimental results on the benchmark data sets in terms of classification accuracy, shown in Tables \ref{Table:threemoonaccuracy}--\ref{Table:Opticsaccuracy}, indicate that our proposed method outperforms the state-of-the-art methods for high-dimensional data and point clouds classification.

Compared to the start-of-the-art variational classification models proposed e.g. in \cite{BM17,YT18}, in addition to the data fidelity term and $\ell_1$ term (e.g. TV), our proposed model in \eqref{eqn:model-proposed} contains an additional $\ell_2$ term on the labeling functions which is used to smooth the classification results so as to reduce the non-smooth artifact (the so-called staircase artifact in images) introduced by the $\ell_1$ term. This is one reason that our method can generally achieve better results. Moreover, the warm initialization used in our method can also play a role to improve the classification quality. Apart from generating the initialization manually, any classification methods can practically be used to generate the initialization. Starting from the initialization, our proposed method can then be applied to achieve a better classification result by improving the accuracy iteratively. Theoretically speaking, the poorer the quality of the initialization, the more iterations are needed for our method. Nevertheless, we found that even for poor initializations (e.g. the ones generated randomly), 20 iterations are already enough to achieve competitive results. Generally, no more than 15 iterations are needed when using an initialization computed by standard classification methods (e.g. SVM).

Another distinction of our proposed model compared to the variational classification models in e.g. \cite{BM17,YT18} is that there are no constraints on these labeling functions in our objective functional. In other words, in each iteration, we just need to find the minimizer of the objective functional corresponding to each labeling function, but these minimizers do not need to satisfy the constraint that their summation equal to 1. Therefore, the computation speed for every single iteration is improved in our method compared to other methods which have constraints. We emphasize again that, since minimizing each sub-problem with respect to each labeling function is irrelevant to minimizing the sub-problems with respect to other labeling functions, parallelism techniques can be used straightforwardly to further improve the computation performance of our algorithm; theoretically, we just require  $1/K$ of the computation time needed for the non-parallelism scheme. This will be extremely important for large data sets.
From Table \ref{Table:time}, we see that, for all the computation time of our method, when considering the effect of parallel computing, our method should be able to outperform the state-of-the-art methods by a large margin.

%-------------------------------------------------------------------
\section{Conclusions}\label{sec:conclusions}
%-------------------------------------------------------------------
In this paper, a two-stage multiphase semi-supervised method is proposed for classifying high-dimensional data or unstructured point clouds. The method is based on the SaT strategy which has been shown very effective for segmentation problems such as gray or color images corrupted by different degradations. Starting with a proper initialization which can be obtained by using any standard classification algorithm (e.g. SVM) or constructed by users, the first stage of our method is to solve a convex variational model
without constraint. Most importantly, our proposed model is a lot easier to solve than the state-of-the-art variational models proposed recently (e.g. \cite{BM17,YT18}) for point clouds classification problem since they all need no vacuum and overlap constraint \eqref{eqn:partition} on the labeling functions in the unit simplex which could make their models to be non-convex. The second stage of our method is to find a binary partition via thresholding the smoothed result obtained from stage one. We proved that our proposed model has a unique solution and the derived primal-dual algorithm converges.

We tested our proposed method on four benchmark data sets and compared with the state-of-the-art methods. We also investigated the influence of the training sets
selected uniformly and non-uniformly. For our method, different ways of generating initializations were implemented and validated. On the whole, the experimental results demonstrated that our method is superior in terms of classification accuracy and when parallel computing is considered, computation speed too. Therefore our method is an efficient and effective classification method for data sets like high-dimensional data or unstructured point clouds.

%-------------------------------------------------------------------
\section*{Acknowledgements}
%-------------------------------------------------------------------
This work of R. Chan is partially supported by HKRGC Grants No. CityU12500915, CityU14306316, HKRGC CRF Grant C1007-15G, and HKRGC AoE Grant AoE/M-05/12. This work of T. Zeng is partially supported by the National Natural Science Foundation of China under Grant 11671002, CUHK start-up and CUHK DAG 4053296, 4053342.
We thank Prof. Xue-Cheng Tai, Dr. Ke Yin, Dr. Egil Bae and Prof. Ekaterina Merkurjev for providing the codes of their methods \cite{BM17,YT18}.

%%-------------------------------------------------------------------
%\section*{Appendix}
%%-------------------------------------------------------------------
%
%\subsection*{Proof of Theorem \ref{thm:unique}}
%
%\subsection*{Proof of Theorem \ref{thm:conv}}
%\xyxie{Add the proof here.}

\bibliographystyle{siam}
% \bibliography{myrefs_cai}

\end{document}